\documentclass[12pt,reqno]{amsart}

\newcommand\version{March 23, 2016}

\usepackage{amsmath,amsfonts,amsthm,amssymb,amsxtra,dsfont}

\usepackage{hyperref}

\newtheorem{theorem}{Theorem}[section]
\newtheorem{proposition}[theorem]{Proposition}
\newtheorem{lemma}[theorem]{Lemma}
\newtheorem{corollary}[theorem]{Corollary}
\newtheorem{conjecture}[theorem]{Conjecture}

\theoremstyle{definition}

\newtheorem{definition}[theorem]{Definition}

\theoremstyle{remark}

\newtheorem{remark}[theorem]{Remark}

\numberwithin{equation}{section}

\newcommand{\cB}{\mathcal{B}}

\newcommand{\C}{\mathbb{C}}

\newcommand{\cO}{\mathcal{O}}

\newcommand{\cR}{\mathcal{R}}
\newcommand{\cS}{\mathcal{S}}

\renewcommand{\epsilon}{\varepsilon}

\newcommand{\F}{\mathcal{F}}
\newcommand{\ii}{\infty}
\newcommand{\loc}{{\rm loc}}

\newcommand{\N}{\mathbb{N}}

\renewcommand{\phi}{\varphi}
\newcommand{\R}{\mathbb{R}}

\newcommand{\Sph}{\mathbb{S}}

\newcommand{\Z}{\mathbb{Z}}
\newcommand\1{{\ensuremath {\mathds 1} }}

\renewcommand{\hat}{\widehat}

\DeclareMathOperator{\re}{Re}

\DeclareMathOperator{\supp}{supp}

\begin{document}

\title[Stein--Tomas inequality --- \version]{Maximizers for the Stein--Tomas inequality}

\author{Rupert L. Frank}
\address{Rupert L. Frank, Mathematics 253-37, Caltech, Pasadena, CA 91125, USA}
\email{rlfrank@caltech.edu}

\author{Elliott H. Lieb}
\address{Elliott H. Lieb, Departments of Mathematics and Physics, Princeton
University, Princeton, NJ 08544, USA}
\email{lieb@princeton.edu}

\author{Julien Sabin}
\address{Julien Sabin, Laboratoire de Math\'ematiques d'Orsay, Univ. Paris-Sud, CNRS, Universit\'e
Paris-Saclay, 91405 Orsay, France.}
\email{Julien.Sabin@math.u-psud.fr}

\begin{abstract}
We give a necessary and sufficient condition for the precompactness of all optimizing sequences for the Stein--Tomas inequality. In particular, if a well-known conjecture about the optimal constant in the Strichartz inequality is true, we obtain the existence of an optimizer in the Stein--Tomas inequality. Our result is valid in any dimension. 
\end{abstract}


\maketitle

\renewcommand{\thefootnote}{${}$} \footnotetext{\copyright\, 2016 by the authors. This paper may be reproduced, in its entirety, for non-commercial purposes.}


\section{Main result}

A fundamental result in harmonic analysis is the Stein--Tomas theorem \cite{Stein-book-84,Tomas-75}, which states that if $f\in L^2(\Sph^{N-1})$, $N\geq 2$, then the inverse Fourier transform $\check f$ of $f\,d\omega$, with $d\omega$ the surface measure on $\Sph^{N-1}$, that is,
$$
\check f(x) := \frac{1}{(2\pi)^{N/2}} \int_{\Sph^{N-1}} e^{ix\cdot\omega} f(\omega)\,d\omega \,,
$$
belongs to $L^q(\R^N)$ with
\begin{equation}
\label{eq:q}
q := 2(N+1)/(N-1)
\end{equation}
and its $L^q(\R^N)$ norm is bounded by a constant times the $L^2(\Sph^{N-1})$ norm of $f$. Moreover, it is well known that the exponent $q$ is optimal (smallest possible) for this to hold for any $f\in L^2(\Sph^{N-1})$.

In this paper we are interested in the optimal Stein--Tomas constant,
$$
\mathcal R_N := \sup_{0\not\equiv f\in L^2(\Sph^{N-1})} \frac{\int_{\R^N} |\check f |^q\,dx}{\|f\|^q} \,,
$$
where $\|\cdot\|$ denotes the norm in $L^2(\Sph^{N-1})$. The value of $\mathcal R_N$ and optimizing functions are only known in dimension $N=3$ due to a remarkable work of Foschi \cite{Foschi-15}; see \cite{CarFosOliThi-15} for partial progress in $N=2$. Our main concern here is whether the supremum defining $\mathcal R_N$ is attained and, more generally, the description of maximizing sequences for $\mathcal R_N$. These questions were recently considered in fundamental papers by Christ and Shao, where the existence of a maximizer for $N=3$ \cite{ChrSha-12a} and $N=2$ \cite{Shao-15} was shown, as well as a precompactness result for maximizing sequences for $N=3$ \cite{ChrSha-12b}. What makes dimensions $N=2$ and $3$ special is that the exponent $q$ in \eqref{eq:q} is an even integer, so that one can multiply out $|\check f|^q$. Our results will be valid in any dimension.

Christ and Shao discovered that for the problem of existence of an maximizer for $\mathcal R_N$ a key role is played by the Strichartz inequality \cite{Strichartz-77}. The optimal constant in this inequality is
$$
\mathcal S_d := (2\pi)^{-(d+2)/d} \sup_{0\neq\psi\in L^2(\R^d)} \frac{\iint_{\R\times\R^d} |e^{it\Delta/2}\psi(x)|^{2+4/d}\,dx\,dt}{\|\psi\|^{2+4/d}} \,.
$$
(Here $\|\cdot\|$ denotes the norm in $L^2(\R^d)$.) Note that $2+4/d=q$ when $d=N-1$. The overall factor $(2\pi)^{-(d+2)/d}$ and the factor $1/2$ in front of the Laplacian are not important, but simplify some formulas below.

We say that a sequence $(f_n)\subset L^2(\Sph^{N-1})$ is precompact in $L^2(\Sph^{N-1})$ \emph{up to modulations} if there is a subsequence $(f_{n_k})$ and a sequence $(a_k)\subset\R^N$ such that $e^{-ia_k\cdot\omega} f_{n_k}$ converges in $L^2(\Sph^{N-1})$.

The following is our main result.

\begin{theorem}\label{main}
Let $N\geq 2$. If
\begin{equation}
\label{eq:gap}
\mathcal R_N > \frac{2^{q/2}}{\sqrt\pi} \frac{\Gamma(\frac{q+1}{2})}{\Gamma(\frac{q+2}{2})}\, \mathcal S_{N-1} \,,
\end{equation}
then maximizing sequences for $\mathcal R_N$, normalized in $L^2(\Sph^{N-1})$, are precompact in $L^2(\Sph^{N-1})$ up to modulations and, in particular, there is a maximizer for $\mathcal R_N$.
\end{theorem}

Clearly, the optimization problem for $\mathcal R_N$ is invariant under modulations, so precompactness up to modulations is the best one can expect. Our theorem says that assumption \eqref{eq:gap} is sufficient for this. In fact, it is easy to see that \eqref{eq:gap} is also necessary for the precompactness modulo modulations of all maximizing sequences. We will comment on this in Remark \ref{gapnotstrict}, where we will also see that \eqref{eq:gap} holds with $\geq$ instead of $>$.

As we will argue below, in dimensions $N=2$ and $N=3$, the strict inequality \eqref{eq:gap} holds and so we recover the Christ--Shao results on the existence of optimizers \cite{ChrSha-12a,Shao-15} and precompactness in $N=3$ \cite{ChrSha-12b} and we obtain, for the first time, precompactness of maximizing sequences for $N=2$.

We believe, but cannot prove, that the strict inequality \eqref{eq:gap} holds in any dimension. To verify it, it seems natural to first compute $\mathcal S_{N-1}$ and then to use a perturbation argument to establish \eqref{eq:gap}. In fact, by a remarkable work of Foschi \cite{Foschi-07} (see also \cite{HunZha-06,BenBezCarHun-09}), the value of $\mathcal S_{N-1}$ is known for $N=2$ and $N=3$. We cite the following conjecture from \cite{Foschi-07}; see also \cite{HunZha-06}.

\begin{conjecture}\label{conj}
Let $d\geq 3$. Then the supremum defining $\mathcal S_d$ is attained for $\psi(x) = e^{-x^2/2}$, $x\in\R^d$.
\end{conjecture}

Assuming that this conjecture is true we can generalize an argument from \cite{ChrSha-12a,Shao-15} and obtain existence of a maximizer for the $\mathcal R_{d+1}$ problem.

\begin{proposition}\label{pert}
Let $N\geq 4$. If Conjecture \ref{conj} holds for $d=N-1$, then \eqref{eq:gap} holds and therefore the conclusions of Theorem \ref{main} hold.
\end{proposition}

In connection with Conjecture \ref{conj} we would like to mention that the existence and precompactness problem for the optimization corresponding to $\mathcal{S}_d$ was solved by Kunze \cite{Kunze-03} in $d=1$ and by Shao \cite{Shao-09} in $d\ge1$. As we will explain next, this problem is considerably easier than that for $\mathcal{R}_N$ since on the paraboloid $\{(\xi,\omega)\in\R^d\times\R:\ |\xi|^2=\omega\}$, no points have parallel normal vectors (which is also a consequence of the fact that the paraboloid can be written globally as a graph and has non-vanishing curvature). In fact, our proof technique allows one to simplify the proofs in \cite{Kunze-03,Shao-09}.

Let us discuss some of the challenges in proving Theorem \ref{main}. As in most optimization problems the key difficulty here is to find a weak limit of an optimizing sequence which is non-zero. There is an obvious way how a maximizing sequence can go weakly to zero, namely by modulations. However, potentially there is another way, namely by concentration and, in fact, the largest part of our proof is concerned with showing that concentration does \emph{not} occur. If a sequence would concentrate at a point, we could approximate the sphere close to this concentration point by a paraboloid and we are in the setting of the Strichartz inequality. (Note that the Strichartz inequality is invariant under dilations.) Therefore, if a maximizing sequence concentrates at a point, one could naively expect that the largest possible `energy' it can have is $\mathcal S_{N-1}$. What makes this problem interesting is that a maximizing sequence can do better than concentrating at a single point! Namely, it can concentrate at a \emph{pair of antipodal points}. What we will show is that the largest possible `energy' in this case is $\frac{2^{q/2}}{\sqrt\pi} \frac{\Gamma(\frac{q+1}{2})}{\Gamma(\frac{q+2}{2})}\, \mathcal S_{N-1}$ with a factor
$$
\frac{2^{q/2}}{\sqrt\pi} \frac{\Gamma(\frac{q+1}{2})}{\Gamma(\frac{q+2}{2})} >1 \,.
$$
From this and our assumption \eqref{eq:gap} we will deduce that maximizing sequences cannot concentrate at two antipodal points and therefore will be precompact.

The fact that a strict `energy' inequality leads to precompactness of minimizing sequences is frequently used in the calculus of variations, for instance, in the linear Schr\"odinger operator theory. In a non-linear context it seems to appear for the first time in the Br\'ezis--Nirenberg problem \cite[Lem. 1.2]{BreNir-83}. (Existence of minimizers, but not precompactness of minimizing sequences, under a strict `energy' inequality was shown earlier in the Yamabe problem \cite{Au}.) We emphasize that both in the Yamabe and in the Br\'ezis--Nirenberg problem one has to deal with the loss of compactness due to concentration around a point.

However, the fact that concentration at two points is better than concentration at a single point is a \emph{non-local phenomenon} and is a novel feature of the optimization problem $\mathcal R_{N}$. As far as we know, it does not appear in optimization problems related to Sobolev spaces (for instance, the Yamabe problem or the Br\'ezis--Nirenberg problem mentioned before -- not even in non-local versions of these problems) or in the optimization problem related to the Strichartz inequality. In order to deal with this non-local effect we have to modify existing strategies in the calculus of variations and we hope that our techniques will be useful in problems with a similar flavor. In particular, our method should allow to solve the case of a general manifold with positive Gauss curvature. In this case the role of antipodal points is played by pairs of points with opposite normal vectors. For earlier results in the case of general curves ($N=2$), but with pairs of points with opposite normal vectors excluded, we refer to \cite{Oliveira-14}.

The mechanism of antipodal concentration was discovered by Christ and Shao in \cite{ChrSha-12a}. In their analysis, however, the fact that $q$ is even plays a major role. First, it allows them to restrict their attention to non-negative functions, which eliminates the loss of compactness due to modulations. More importantly, however, it also allows them to restrict their attention to antipodally symmetric functions. In this way the concentration at antipodal points is built into their proof automatically and, for instance, it is trivial in their case that the concentration happens with the same profile at both points, whereas this is a non-trivial step in our proof.

In order to prove Theorem \ref{main} we use the \emph{method of the missing mass} (MMM) which was invented in \cite{Lieb-83b} and \cite[Lem.~1.2]{BreNir-83}; see also \cite{BreLie-84,FroLieLos-86} for early and \cite{FraLie-12,FraLie-15} for more recent applications of this method. The basic idea is to decompose a maximizing sequence into a main piece, which converges in a strong sense, and a remainder piece, which vanishes in a suitable sense. The goal of the decomposition is that each of the quantities involved in the maximization problem splits into a contribution of the main piece and the remainder piece, without any interaction between them. The crucial point is to not ignore the remainder piece (i.e., the missing mass), but to treat it as a potential optimizer. Because of the non-linear nature of the optimization problem, one can then conclude that the missing mass is either everything (which is impossible, since the main piece does not vanish) or nothing, which means that the maximizing sequence converges, in fact, strongly.

The MMM can deal both with exact symmetries (as in \cite{BreLie-84}) and with almost symmetries (as in \cite{BreNir-83}). One novelty of our work is that we need to apply the method twice, once to deal with the exact modulation symmetry (Proposition \ref{step1}) and once to deal with the almost dilation symmetry (Proposition \ref{step2}).

The method relies on two main ingredients which have to be verified in each problem. First, one needs to identify a main piece which does not vanish in the limit. This usually comes from a compactness theorem. In our case we prove a refinement of the Stein--Tomas inequality (Proposition \ref{steintomref}) which relies on a deep bilinear restriction theorem of Tao \cite{Tao-03}. Our strategy here is reminiscent of Tao's proof of what he calls the `inverse Strichartz theorem' \cite{Tao-09}. We feel that this approach is more direct than earlier approaches using $X_{p}$ spaces, which were used in connection with refined Strichartz inequalities (and were also an ingredient in \cite{Kunze-03,Shao-09}). Refinements of the Stein--Tomas inequality in terms of these spaces also play an important role in the works of Christ--Shao \cite{ChrSha-12a} and Shao \cite{Shao-15}.

The second ingredient in the MMM is the decoupling of the main and the remainder piece. While for a Hilbertian norm  involved in the maximization problem this follows simply from weak convergence, one usually uses almost everywhere convergence and the Br\'ezis--Lieb lemma \cite{Lieb-83b,BreLie-83} for an $L^q$ norm. Indeed, we are able to verify almost everywhere convergence in our setting by proving an analogue of the local smoothing property of the Schr\"odinger equation (Lemma \ref{locsmooth}). However, we need a generalization of the Br\'ezis--Lieb lemma (Lemma \ref{bl}) since in our second application of MMM the main piece will \emph{not} be convergent. Nevertheless, we will be able to separate its contribution from that of the remainder piece. The fact that the main piece is not convergent is ultimately a consequence of the non-local interaction between concentration points.

The outline of this paper is as follows. In Section \ref{sec:mmm} we present the overall strategy of our argument in more detail and explain how the MMM works. Section \ref{sec:bl} contains the new generalization of the Br\'ezis--Lieb lemma, Section \ref{sec:top} the results on almost everywhere convergence and Section \ref{sec:comp} (and Appendix \ref{sec:strichref}) the compactness result mentioned before. In Section \ref{sec:equal} we complete the computation of the compactness level by showing that, if concentration at antipodal points happens, then it is energetically favorable to have the same concentration profile on both points. Finally, Section \ref{sec:pert} is devoted to the proof of Proposition \ref{pert}.

\subsection*{Acknowledgement}

R.L.F. and J.S. would like to thank D. Oliveira e Silva and C. Thiele for the summer school `Sharp inequalities in harmonic analysis' in August 2015 which stimulated our interest in this project. Partially support by U.S. National Science Foundation DMS-1363432 (R.L.F.) and PHY-1265118 (E.H.L.) is acknowledged.


\section{Outline of the proof. Method of the missing mass}\label{sec:mmm}

In this section we explain the main steps in the proof of Theorem \ref{main}. In Proposition~\ref{step1} we will show that the conclusions of Theorem \ref{main} hold if $\frac{2^{q/2}}{\sqrt\pi} \frac{\Gamma(\frac{q+1}{2})}{\Gamma(\frac{q+2}{2})}\, \mathcal S_{N-1}$ on the right side of \eqref{eq:gap} is replace by a certain quantity $\mathcal R_N^*$, which is abstractly defined through certain sequences in $L^2(\Sph^{N-1})$ that converge weakly to zero. In a second step in Proposition \ref{step2} we will show that
\begin{equation}
\label{eq:step2}
\mathcal R_N^* = \tilde{\mathcal S}_{N-1} \,,
\end{equation}
where $\tilde{\mathcal S}_{N-1}$ is a quantity defined in terms of pairs of functions in $L^2(\R^{N-1})$ and is a generalization of the Strichartz constant $\mathcal S_{N-1}$. Finally, in Section \ref{sec:equal} we will show that
\begin{equation}
\label{eq:step3}
\tilde{\mathcal S}_d = \frac{2^{q/2}}{\sqrt\pi} \frac{\Gamma(\frac{q+1}{2})}{\Gamma(\frac{q+2}{2})}\, \mathcal S_d \,,
\end{equation}
which will complete the proof of Theorem \ref{main}.

We now present these steps in more detail.

\begin{definition}\label{mod}
Let $(f_n)\subset L^2(\Sph^{N-1})$. We write
$$
f_n\rightharpoonup_{\text{mod}} 0
$$
if for every sequence $(a_n)\subset\R^N$ one has
$$
e^{-ia_n\cdot\omega} f_n \rightharpoonup 0
\qquad\text{in}\ L^2(\Sph^{N-1}) \,.
$$
\end{definition}

(Here and in the following, we slightly abuse notation and write $e^{-ia_n\cdot\omega} f$ for the function $\omega\mapsto e^{-ia_n\cdot\omega} f(\omega)$.)

Define
$$
\mathcal R_N^* := \sup \left\{ \limsup \int_{\R^N} |\check f_n|^q\,dx :\ \|f_n\|=1 \,,\ f_n \rightharpoonup_{\text{mod}} 0 \right\}.
$$

\begin{proposition}\label{step1}
If
\begin{equation*}
\mathcal R_N > \mathcal R_N^* \,,
\end{equation*}
then maximizing sequences for $\mathcal R_N$, normalized in $L^2(\Sph^{N-1})$, are precompact in $L^2(\Sph^{N-1})$ up to modulations and, in particular, there is a maximizer for $\mathcal R_N$.
\end{proposition}

\begin{proof}
Let $(f_n)\subset L^2(\Sph^{N-1})$ be a maximizing sequence with $\|f_n\|=1$. Since
$$
\lim \int_{\R^N} |\check f_n|^q\,dx = \mathcal R_N >\mathcal R_N^* \,,
$$
we infer that $f_n \not \rightharpoonup_{\text{mod}} 0$. That is, there is an $h\in L^2(\Sph^{N-1})$ and a sequence $(a_n)\subset\R^N$ such that $\limsup_{n\to\infty} \left| \int h e^{-ia_n\cdot\omega} f_n \,d\omega\right|>0$. After passing to a subsequence we may assume that $\inf_n \left| \int h e^{-ia_n\cdot\omega} f_n \,d\omega\right|>0$. By weak compactness, after passing to another subsequence, we may assume that $e^{-ia_n\cdot\omega} f_n\rightharpoonup g$ in $L^2(\Sph^{N-1})$. Then $\int h e^{-ia_n\cdot\omega} f_n \,d\omega\to \int h g\,d\omega$ and this is non-zero, so we conclude that $g\not\equiv 0$.

Let us denote $r_n:= e^{-ia_n\cdot\omega} f_n - g$. Then $r_n\rightharpoonup 0$ in $L^2(\Sph^{N-1})$ and therefore
$$
m:= \lim_{n\to\infty} \|r_n\|^2
\qquad\text{exists and satisfies}\ 1=\|g\|^2 + m \,.
$$
Moreover, since $e^{ix\cdot\omega}\in L^2(\Sph^{N-1})$, weak convergence implies that $\check r_n\to 0$ pointwise and therefore, by the Br\'ezis--Lieb lemma \cite{Lieb-83b,BreLie-83},
$$
\mu := \lim_{n\to\infty} \| \check r_n\|_q^q
\qquad\text{exists and satisfies}\ \mathcal R_N=\|\check g\|^q_q + \mu \,.
$$
Since $\| \check r_n\|_q^q \leq \mathcal R_N \|r_n\|^q$, we have $\mu\leq\mathcal R_N m^{q/2}$ and therefore
$$
\mathcal R_N = \|\check g\|^q_q + \mu \leq \|\check g\|^q_q + \mathcal R_N m^{q/2} = \|\check g\|^q_q + \mathcal R_N (1-\|g\|^2)^{q/2} \leq \|\check g\|^q_q + \mathcal R_N -\mathcal R_N \|g\|^q \,, 
$$
where we used the elementary inequality $(1-t)^{q/2}\leq 1-t^{q/2}$ for $t\in[0,1]$. Thus, we have shown that $0\leq \|\check g\|^q_q -\mathcal R_N \|g\|^q$, which means that $g$ is a maximizer (recall that $g\not\equiv 0$) and that equality must hold everywhere. Since the elementary inequality is strict unless $t\in\{0,1\}$, we conclude that $\|g\|^2=1$. Thus, $m=0$, which means that $e^{-ia_n\cdot\omega} f_n$ converges to $g$ strongly in $L^2(\Sph^{N-1})$. This completes the proof.
\end{proof}

This proposition reduces the proof of our main theorem to showing that
$$
\mathcal R_N^* = \frac{2^{q/2}}{\sqrt\pi} \frac{\Gamma(\frac{q+1}{2})}{\Gamma(\frac{q+2}{2})}\, \mathcal S_{N-1} \,,
$$
which we will verify in two steps. Let
\begin{align*}
& \tilde{\mathcal S}_d := (2\pi)^{-(d+2)/d} \\
& \quad \times \!\!\!\!\sup_{(0,0)\neq(\psi^+,\psi^-)\in L^2(\R^d)^2} \frac{\lim_{\lambda\to\infty} \iint_{\R\times\R^d} \left| e^{it\Delta/2}\psi^+(x) + e^{i\lambda x_N} e^{-it\Delta/2}\psi^-(x)\right|^{2+4/d}\,dx\,dt}{\left( \|\psi^+\|^2 + \|\psi^-\|^2\right)^{1+2/d}}.
\end{align*}
It is easy to see that the limit $\lambda\to\infty$ exists. We discuss this in some more detail before Lemma \ref{equal}.

Our next goal is to prove equality \eqref{eq:step2}. Intuitively, this equality says that for the computation of $\mathcal R_N^*$ we only need to consider sequences which concentrate on a pair of antipodal points. Approximating the sphere near the concentration points by a paraboloid, we arrive at $\tilde{\mathcal S}_{N-1}$. (The factor of $(2\pi)^{-(d+2)/d}$ comes from the normalization of the Fourier transform.)

In order to make this intuition precise we have to quantify the notion of concentration. We will introduce a family of maps $\mathcal B_{R,\delta}$ with $R\in\cO(N)$ and $\delta>0$ which identifies pairs of functions on $L^2(\R^{N-1})$ with a function on $L^2(\Sph^{N-1})$. The orthogonal matrix $R\in\cO(N)$ will determine the equator along which we cut the function in $L^2(\Sph^{N-1})$ into two pieces. The parameter $\delta>0$ corresponds to a scaling in $L^2(\R^{N-1})$.

We begin with the case $R=\text{Id}$, in which the equator along which we cut is the standard equator. For $\phi^+,\phi^-\in L^2(\R^{N-1})$ and $\delta>0$ we define a function $\mathcal B_{\delta}(\phi^+,\phi^-)\in L^2(\Sph^{N-1})$ by
\begin{equation}\label{eq:defi-Bn}
\begin{cases}
    \displaystyle  \mathcal B_{\delta}(\phi^+,\phi^-)\left(\frac{\xi}{\sqrt{1+\xi^2}},\frac{1}{\sqrt{1+\xi^2}}  \right) := (1+\xi^2)^{N/4}\,\delta^{-(N-1)/2}\, \phi^+(\xi/\delta),\\
    \displaystyle  \mathcal B_{\delta}(\phi^+,\phi^-)\left(\frac{\xi}{\sqrt{1+\xi^2}},\frac{-1}{\sqrt{1+\xi^2}}\right) := (1+\xi^2)^{N/4} \, \delta^{-(N-1)/2}\, \phi^-(\xi/\delta)
  \end{cases} 
\end{equation}
for $\xi\in\R^{N-1}$. (It is inessential that $B_\delta(\phi^+,\phi^-)$ is not defined on the set $\{\omega\in\Sph^{N-1}:\ \omega_N=0\}$ of measure zero.) A simple change of variables shows that
\begin{equation}
\label{eq:bunitary}
\|\mathcal B_{\delta}(\phi^+,\phi^-)\|^2 = \|\phi^+\|^2 + \|\phi^-\|^2 \,.
\end{equation}
The map $\mathcal B_\delta$ will be $\mathcal B_{R,\delta}$ with $R=\text{Id}$. Now for any $R\in\cO(N)$,  $\phi^+,\phi^-\in L^2(\R^{N-1})$, and $\delta>0$ we define a function $\mathcal B_{R,\delta}(\phi^+,\phi^-)\in L^2(\Sph^{N-1})$ by
\begin{equation}
\label{eq:defi-B}
\mathcal B_{R,\delta}(\phi^+,\phi^-)(\omega) = \mathcal B_{\delta}(\phi^+,\phi^-)(R^{-1}\omega) \,.
\end{equation}
Since $\cB_\delta$ concentrates as $\delta\to 0$ around the north pole $(0,\ldots,0,1)$ and the south pole $(0,\ldots,0,-1)$, $\cB_{R,\delta}$ concentrates around $R(0,\ldots,0,1)$ and $R(0,\ldots,0,-1)$ as $\delta\to0$. 

\begin{definition}\label{conc}
Let $(f_n)\subset L^2(\Sph^{N-1})$. We write
$$
f_n\rightharpoonup_{\text{conc}} 0
$$
if for all sequences $(a_n)\subset\R^N$, $(R_n)\subset\cO(N)$ and $(\delta_n)\subset (0,\infty)$ with $\sup\delta_n<\infty$ one has
$$
\mathcal B_{R_n,\delta_n}^{-1} \left( e^{-ia_n\cdot\omega} f_n \right) \rightharpoonup (0,0)
\qquad\text{in}\ L^2(\R^{N-1})\times L^2(\R^{N-1}) \,.
$$
\end{definition}

(Recall that, with our slight abuse of notation, $e^{-ia_n\cdot\omega} f$ denotes the function $\omega\mapsto e^{-ia_n\cdot\omega} f(\omega)$.)

Let us briefly comment on this definition. At first sight it might look unnecessary to include a sequence of orthogonal maps $(R_n)$ in this definition since the space $\cO(N)$ is compact and hence, up to a subsequence, $(R_n)$ will converge to a fixed orthogonal map. However, if $\delta_n\to0$, the sphere gets `blown-up' and the maps $(R_n)$ might move a point on the sphere on a distance $1/\delta_n$ when looking around the concentration point at the scale $1/\delta_n$. As a consequence, the $(R_n)$ play the role of the $v$-translations (modulation symmetry) in the symmetries of the Strichartz inequality (see the appendix of \cite{Tao-09}). The importance of keeping these rotations will become clear in the proof of Lemma \ref{lem:Linfty-estimate-Tao}. Let us also remark that the analogue of our $(a_n)$ are $(t_n,x_n)$-translations in the Strichartz case.

Our definition of the convergence $f_n\rightharpoonup_{\text{conc}} 0$ is specific to the sphere: we used that any rotation stabilizes the sphere. If one tries to adapt our approach to a general compact manifold with positive Gauss curvature one probably needs to work with local versions of the $\cB$ operators.

\medskip

We introduce two auxiliary functions $\zeta_1,\zeta_2$ on $[0,\infty)$ by
$$
\zeta_1(k) = \frac{1}{\sqrt{1+k^2}} \,,
\qquad
\zeta_2(k) = \frac{2}{k^2} \left( 1- \frac{1}{\sqrt{1+k^2}} \right).
$$
For $\psi\in L^2(\R^{N-1})$ we define with $x=(x',x_N)\in\R^{N-1}\times\R$
\begin{equation}
\label{eq:deft}
\left(\mathcal T_\delta\psi\right)(x) := \frac{1}{(2\pi)^{(N-1)/2}} \int_{\R^{N-1}} \hat\psi(\xi) e^{i\left( \xi\cdot x' \zeta_1(\delta|\xi|) - \frac12 \xi^2 x_N \zeta_2(\delta|\xi|)\right)} \,\frac{d\xi}{(1+\delta^2\xi^2)^{N/4}} \,.
\end{equation}
The operators $\mathcal T_\delta$ arise naturally in this context since for any pair of functions $\psi^+,\psi^-\in L^2(\R^{N-1})$ and any $\delta>0$, setting
$$
f := \mathcal B_\delta (\widehat{\psi^+},\widehat{\psi^-}) \,,
$$
we find
\begin{align}
\label{eq:bt}
\delta^{-(N-1)/2} \check f(x'/\delta,x_N/\delta^2) = (2\pi)^{-1/2} & \left( e^{ix_N/\delta^2} \left( \mathcal T_\delta \psi^+\right)(x) + \, e^{-ix_N/\delta^2} \left( \mathcal T_\delta \psi^-\right)(x',-x_N) \right).
\end{align}
This follows by a simple change of variables.

We are now able to carry out the second step in the proof of Theorem \ref{main}, which is a variation of the argument used to prove Proposition \ref{step1} combined with a compactness theorem for the $f_n\rightharpoonup_{\text{conc}}0$ convergence (Corollary \ref{concentration}) and two convergence theorems for the operators $\mathcal T_\delta$ (Propositions \ref{tconv} and \ref{aeconv}).

\begin{proposition}\label{step2}
$\mathcal R_N^* = \tilde{\mathcal S}_{N-1}$
\end{proposition}

\begin{proof}
We begin with the proof of $\leq$. Let $(f_n)\subset L^2(\Sph^{N-1})$ be a sequence with $\|f_n\|=1$, $f_n\rightharpoonup_{\text{mod}} 0$ and $\|\check f_n\|_q^q \to\mathcal R_N^*$. We may assume that $\mathcal R_N^*>0$, for otherwise there is nothing to prove, and therefore $\check f_n\not\to 0$ in $L^q(\R^N)$. According to Corollary \ref{concentration} and weak compactness, after passing to a subsequence, we may assume that there are sequences $(a_n)\subset\R^N$, $(R_n)\subset\cO(N)$ and $(\delta_n)\subset(0,\infty)$ with $\sup\delta_n<\infty$ and functions $\psi^+,\psi^-\in L^2(\R^{N-1})$ with
\begin{equation}
\label{eq:step2nonzero}
\|\psi^+\|^2 + \|\psi^-\|^2 \neq 0
\end{equation}
such that $\mathcal B_{R_n,\delta_n}^{-1} \left( e^{-ia_n\cdot\omega} f_n \right) \rightharpoonup (\widehat{\psi^+},\widehat{\psi^-})$ in $L^2(\R^{N-1})\times L^2(\R^{N-1})$. Since $f_n\rightharpoonup_{\text{mod}} 0$, we have $\delta_n\to 0$. Because of rotation and modulation invariance of the problem, we may assume that $R_n=\text{Id}$ and $a_n=0$ for all $n$ and we write $\mathcal B_{\delta_n}$ instead of $\mathcal B_{R_n,\delta_n}$.

We define
$$
p_n:= \mathcal B_{\delta_n}(\widehat{\psi^+},\widehat{\psi^-})
\qquad\text{and}\qquad
r_n:= f_n - \mathcal B_{\delta_n}(\widehat{\psi^+},\widehat{\psi^-}) \,.
$$
We shall show that
\begin{equation}
\label{eq:mm1}
m := \lim_{n\to\infty} \|r_n\|^2
\qquad\text{exists and satisfies}\
1= \|\psi^+\|^2 + \|\psi^-\|^2 + m
\end{equation}
and
\begin{equation}
\label{eq:mm2}
\mu := \lim_{n\to\infty} \int_{\R^N} |\check r_n|^q\,dx
\qquad\text{exists and satisfies}\
\mathcal R_N^* = P + \mu \,,
\end{equation}
where
$$
P := \lim_{n\to\infty} \int_{\R^N} |\pi_n|^q\,dx
$$
and
$$
\pi_n(x) := (2\pi)^{-1/2} \left( e^{ix_N/\delta_n^2} \left( e^{ix_N\Delta/2}\psi^+\right)(x') + e^{-ix_N/\delta_n^2} \left( e^{-ix_N\Delta/2}\psi^-\right)(x') \right) \,.
$$
(The fact that the limit definining $P$ exists is again a consequence of the arguments before Lemma \ref{equal}. In fact, we do not really need here the existence of the limit, but could simply work with the limsup in the definitions of both $\mu$ and $P$.)

Before proving \eqref{eq:mm1} and \eqref{eq:mm2}, let us show that they imply the proposition. Since $\delta_n\to0$ we have $\mathcal B_{\delta_n}(\widehat{\psi^+},\widehat{\psi^-})\rightharpoonup_{\text{mod}}0$, and since $f_n\rightharpoonup_{\text{mod}}0$ by assumption, we have $r_n\rightharpoonup_{\text{mod}} 0$. Thus,
\begin{equation}
\label{eq:mm3}
\mu \leq \mathcal R_N^*\, m^{q/2} \,.
\end{equation}
(In fact, if $m=0$, this follows from the Stein--Tomas inequality and, if $0<m<1$, it follows by using the definition of $\mathcal R_N^*$ for the sequence $r_n/\|r_n\|$.) Combining \eqref{eq:mm1}, \eqref{eq:mm2} and \eqref{eq:mm3} and recalling the elementary inequality used in the proof of Proposition \ref{step1} we obtain
\begin{align*}
\mathcal R_N^* & = P + \mu \leq P + \mathcal R_N^* m^{q/2} = P + \mathcal R_N^* \left( 1 - \|\psi^+\|^2 - \|\psi^-\|^2\right)^{q/2} \\
& \leq P + \mathcal R_N^* \left( 1 - \left( \|\psi^+\|^2 + \|\psi^-\|^2\right)^{q/2} \right),
\end{align*}
that is,
$$
\mathcal R_N^* \left( \|\psi^+\|^2 + \|\psi^-\|^2\right)^{q/2} \leq P \,.
$$
Because of \eqref{eq:step2nonzero} this is the claimed upper bound on $\mathcal R_N^*$.

It remains to prove \eqref{eq:mm1} and \eqref{eq:mm2}. For the proof of \eqref{eq:mm1} we recall the unitarity relation \eqref{eq:bunitary} for $\mathcal B_{\delta_n}$. Thus, the weak convergence $\mathcal B_{\delta_n}^{-1}r_n\rightharpoonup 0$ implies
$$
1 = \left\|f_n\right\|^2 = \left\|\mathcal B_{\delta_n}^{-1} f_n \right\|^2 = \left\| (\widehat{\psi^+},\widehat{\psi^-}) + \mathcal B_{\delta_n}^{-1}r_n \right\|^2 = \left\|(\widehat{\psi^+},\widehat{\psi^-})\right\|^2 + \left\|\mathcal B_{\delta_n}^{-1}r_n\right\|^2 + o(1) \,.
$$
Using once again $\|\mathcal B_{\delta_n}^{-1}r_n\|^2 = \|r_n\|^2$, we obtain \eqref{eq:mm1}.

For the proof of \eqref{eq:mm2} we denote $(\widehat{\psi_n^+},\widehat{\psi_n^-}):=\mathcal B_{\delta_n}^{-1} f_n$ and decompose, using \eqref{eq:bt},
\begin{align*}
\delta_n^{-(N-1)/2} \check f_n(x'/\delta_n,x_N/\delta_n^2) 
= \pi_n(x) + \rho_n(x) + \sigma_n(x) \,,
\end{align*}
where we have set
$$
\rho_n(x) := (2\pi)^{-1/2} \left( e^{ix_N/\delta_n^2} \rho_n^+(x',x_N) + e^{-ix_N/\delta_n^2} \rho_n^-(x',-x_N) \right)
$$
with
$$
\rho_n^\pm(x) := \mathcal T_{\delta_n}\left(\psi_n^\pm -\psi^\pm\right)(x)
$$
and
$$
\sigma_n(x) := (2\pi)^{-1/2} \left( e^{ix_N/\delta_n^2} \sigma_n^+(x',x_N) + e^{-ix_N/\delta_n^2} \sigma_n^-(x',-x_N) \right)
$$
with
$$
\sigma_n^\pm(x) := \left(\mathcal T_{\delta_n}\psi^\pm\right)(x) - \left( e^{ix_N\Delta/2}\psi^\pm\right)(x') \,.
$$
It follows from Proposition \ref{aeconv} that, after passing to a subsequence if necessary, $\rho_n^\pm\to 0$ almost everywhere and from Proposition \ref{tconv} that $\sigma_n^\pm \to 0$ in $L^q$. Moreover,
$$
|\pi_n(x)| \leq (2\pi)^{-1/2} \left( \left| \left( e^{ix_N\Delta/2}\psi^+\right)(x')\right| + \left| \left( e^{-ix_N\Delta/2}\psi^-\right)(x') \right| \right) \in L^q_x(\R^N) \,.
$$
Therefore, the generalized Br\'ezis--Lieb Lemma \ref{bl} implies
\begin{align*}
\int_{\R^N} \left| \delta_n^{-(N-1)/2} \check f_n(x'/\delta_n,x_N/\delta_n^2) \right|^q dx = & \int_{\R^N} |\pi_n(x)|^q\,dx + \int_{\R^N} |\rho_n(x)|^q\,dx + o(1) \,.
\end{align*}
By scaling, the left side equals $\|\check f_n\|_q^q$ and, since
$$
\rho_n(x) = \delta_n^{-(N-1)/2} \check r_n(x'/\delta_n,x_N/\delta_n^2) \,,
$$
the second term on the right side equals $\|\check r_n\|_q^q$. Taking the limit as $n\to\infty$ and using the fact that the limit definining $P$ exists we obtain \eqref{eq:mm3}. This completes the proof of the inequality $\leq$ in the proposition.

The proof of the inequality $\geq$ is similar, but simpler. Indeed, pick any pair of functions $(\psi^+,\psi^-)\in L^2(\R^{N-1})^2$ such that $\|\psi^+\|^2+\|\psi^-\|^2=1$ and any sequence $(\delta_n)$ of positive numbers converging to zero. Then, the sequence $f_n:=\mathcal B_{\delta_n}(\widehat{\psi^+},\widehat{\psi^-})$ satisfies $\|f_n\|=1$ and $f_n\rightharpoonup_{\text{mod}}0$. As a consequence,
$$\limsup_{n\to\ii}\int_{\R^N}|\check{f_n}|^q\,dx\le\cR_N^*.$$
On the other hand, by the same method as in the proof of the inequality $\leq$, we have
\begin{align*}
& \liminf_{n\to\ii}\int_{\R^N}|\check{f_n}|^q\,dx \\
& \quad \geq (2\pi)^{-q/2} \lim_{n\to\ii} \int_{\R^N}\left|\left(e^{ix_N\Delta/2}\psi^+\right)(x')+e^{-2ix_N/\delta_n^2} \left( e^{-ix_N\Delta/2}\psi^-\right)(x')\right|^q\,dx \,,
\end{align*}
showing that $\tilde{\mathcal{S}}_{N-1}\le\cR_N^*$.
\end{proof}

To complete the proof of Theorem \ref{main} it suffices to show equality \eqref{eq:step3}. This is the content of Corollary \ref{step3}.

\begin{remark}\label{gapnotstrict}
Similar arguments to those used before show that
\begin{equation}
\label{eq:gapnotstrict}
\mathcal R_N \geq \frac{2^{q/2}}{\sqrt\pi} \frac{\Gamma(\frac{q+1}{2})}{\Gamma(\frac{q+2}{2})}\, \mathcal S_{N-1} \,,
\end{equation}
which is the non-strict version of \eqref{eq:gap}. In fact, we clearly have $\mathcal R_N \geq \mathcal R_N^*$, so that \eqref{eq:gapnotstrict} follows from Proposition \ref{step2} and \eqref{eq:step3}. Moreover, by definition there is a sequence $(f_n)\subset L^2(\Sph^{N-1})$ with $\|f_n\|=1$ and $\|\check f_n\|_q^q \to \mathcal R_N^*$ which is \emph{not} precompact in $L^2(\Sph^{N-1})$. Thus, the \emph{strict} inequality \eqref{eq:gap} is necessary for the precompactness of all maximizing sequences.
\end{remark}


\section{A generalization of the Br\'ezis--Lieb lemma}\label{sec:bl}

The following abstract lemma decouples the main piece from a remainder piece that converges to zero almost everywhere.

\begin{lemma}\label{bl}
Let $(X,dx)$ be a measure space and $p>0$. Let $(\alpha_n)$ be a bounded sequence in $L^p(X)$ such that
$$
\alpha_n = \pi_n + \rho_n + \sigma_n \,,
$$
where, for some $\Pi\in L^p(X)$,
$$
|\pi_n| \leq \Pi
\qquad\text{for all}\ n
$$
and where
$$
\rho_n \to 0 \ \text{almost everywhere}
\qquad\text{and}\qquad
\sigma_n \to 0 \ \text{in} \ L^p(X) \,.
$$
Then
$$
\int_X |\alpha_n|^p\,dx = \int_X |\pi_n|^p\,dx + \int_X |\rho_n|^p\,dx + o(1)
\qquad\text{as}\ n\to\infty \,.
$$
\end{lemma}

Note that, if $\pi_n$ is independent of $n$ and $\sigma_n=0$, this is the result from \cite{Lieb-83b} which was generalized in \cite{BreLie-83}. Our lemma follows by similar arguments as in \cite{BreLie-83}.

\begin{proof}
As a preliminary step we show that the asymptotics are independent of $\sigma_n$. By the triangle inequality we have for $p\geq 1$,
$$
\left| \left\|\alpha_n\right\|_p - \left\|\pi_n+\rho_n\right\|_p \right| \leq \|\sigma_n\|_p
$$
and for $0<p\leq 1$
$$
\left| \left\|\alpha_n\right\|_p^p - \left\|\pi_n+\rho_n\right\|_p^p \right| \leq \|\sigma_n\|_p^p \,.
$$
We conclude that
$$
\int_X |\alpha_n|^p\,dx = \int_X |\pi_n+\rho_n|^p\,dx + o(1) \,.
$$
(For $p>1$ we also use the fact that $\sup_n\|\alpha_n\|_p$ and $\sup_n\|\pi_n+\rho_n\|_p$ are finite; see the proof below.) Thus, the lemma will follow if we can prove that
\begin{equation}
\label{eq:blgoal}
\int_X \left| \left| \pi_n+\rho_n\right|^p - \left|\pi_n\right|^p - \left|\rho_n\right|^p \right| dx =o(1)
\qquad\text{as}\ n\to\infty \,.
\end{equation}

Let $\epsilon>0$ and put
$$
R_n := \left( \left| \left| \pi_n+\rho_n\right|^p - \left|\pi_n\right|^p - \left|\rho_n\right|^p \right| - \epsilon \left|\rho_n\right|^p \right)_+ \,.
$$
Then
$$
\int_X \left| \left| \pi_n+\rho_n\right|^p - \left|\pi_n\right|^p - \left|\rho_n\right|^p \right| dx \leq \epsilon \int_X \left|\rho_n\right|^p \,dx + \int_X R_n\,dx \,,
$$
and asymptotics \eqref{eq:blgoal} will follow if we can prove that
\begin{equation}
\label{eq:blproof0}
\limsup_{n\to\infty} \int_X \left|\rho_n\right|^p \,dx <\infty
\end{equation}
and
\begin{equation}
\label{eq:blproof}
\int_X R_n \,dx = o(1)
\qquad\text{as}\ n\to\infty\,.
\end{equation}

For the proof of \eqref{eq:blproof0} we simply bound
$$
|\rho_n|^p \leq C_p \left( |\alpha_n|^p + |\pi_n|^p + |\sigma_n|^p \right) \leq C_p \left( |\alpha_n|^p + \Pi^p + |\sigma_n|^p \right)
$$
with $C_p =3^{p-1}$ if $p\geq 1$ and $C_p = 1$ if $p<1$. Thus, by assumption,
$$
\limsup_{n\to\infty} \int_X |\rho_n|^p\,dx \leq C_p \limsup_{n\to\infty} \int_X \left( |\alpha_n|^p + \Pi^p \right)dx <\infty \,,
$$
which gives \eqref{eq:blproof0}.

We will prove \eqref{eq:blproof} by dominated convergence. Clearly, there is a $C_{\epsilon,p}$ such that for all $a,b\in\C$,
$$
\left| \left|a+b\right|^p - \left|b\right|^p \right| \leq \epsilon \left|b\right|^p + C_{\epsilon,p} \left|a\right|^p \,.
$$
Thus,
$$
\left| \left| \pi_n+\rho_n\right|^p - \left|\pi_n\right|^p - \left|\rho_n\right|^p \right| 
\leq \left| \left| \pi_n+\rho_n\right|^p - \left|\rho_n\right|^p \right|  + \left|\pi_n\right|^p 
\leq 
\epsilon \left|\rho_n\right|^p + \left( C_{\epsilon,p} + 1 \right) \left|\pi_n\right|^p
$$
and so
$$
R_n \leq \left( C_{\epsilon,p} + 1 \right) \left|\pi_n\right|^p \leq \left( C_{\epsilon,p} + 1 \right) \Pi^p \,.
$$
By assumption, the right side is integrable.

To complete the proof we show that $R_n\to 0$ almost everywhere. Note that $\rho_n\to 0$ almost everywhere and that $|\pi_n|\leq \Pi$. The set $\{\rho_n\to 0\}\cap\{\Pi<\infty\}$ has full measure and on this set we have $R_n\to 0$ almost everywhere. This simply follows from the fact that for sequences $(a_n),(b_n)\subset\C$ with $\sup|a_n|<\infty$ and $b_n\to 0$, we have $|a_n+b_n|^p-|a_n|^p\to 0$ for any $p>0$. This proves the lemma.
\end{proof}


\section{Some a-priori estimates and convergence results}\label{sec:top}

In this section we discuss the convergence properties of the operators $\mathcal T_\delta$ from \eqref{eq:deft} as $\delta\to 0$. These properties were used in the proof of Proposition \ref{step2}. As we have already seen in Section \ref{sec:mmm}, the operators $\mathcal T_\delta$ appear naturally in our problem for functions on $\Sph^{N-1}$ which concentrate near the north pole with the parameter $\delta$ denoting the scale on which the functions live.

\subsection{$L^q$ convergence}

We recall that we always assume $q=2(N+1)/(N-1)$. The purpose of this subsection is to prove the following convergence result.

\begin{proposition}\label{tconv}
Let $\psi\in L^2(\R^{N-1})$. Then, as $\delta\to 0$,
$$
\left(\mathcal T_\delta\psi \right)(x) \to \left( e^{ix_N\Delta/2}\psi \right)(x')
\ \text{in}\ L^q(\R^N) \,.
$$
\end{proposition}

We begin with an a-priori bound for $\mathcal T_\delta$.

\begin{lemma}\label{tbounded}
$\mathcal T_\delta$ is a bounded operator from $L^2(\R^{N-1})$ to $L^q(\R^N)$ and $\|\mathcal T_\delta\|_{L^2\to L^q}$ is independent of $\delta>0$.
\end{lemma}

\begin{proof}[Proof of Lemma \ref{tbounded}]
We claim that for all $\psi\in L^2(\R^{N-1})$ and for all $x\in\R^N$,
\begin{equation}
\label{eq:tfactorization}
(\mathcal T_\delta\psi)(x) =(2\pi)^{1/2}e^{-ix_N/\delta^2}(\mathcal V_{\delta} \mathcal F^{-1} \mathcal B \mathcal U_\delta \mathcal F\psi)(x) \,.
\end{equation}
where $\mathcal V_{\delta}$ and $\mathcal U_\delta$ are isometric isomorphisms in $L^q(\R^N)$ and $L^2(\R^N)$, respectively, $\mathcal F$ denotes the Fourier transform and $\mathcal B$ is a unitary operator from $L^2(\R^{N-1})$ to $L^2(\Sph^{N-1}_+)$ ($\Sph^{N-1}_+$ denoting the northern hemisphere). Thus,
$$
\|\mathcal T_\delta\|_{L^2\to L^q} = (2\pi)^{1/2}\|\mathcal F^{-1}\|_{L^2(\Sph^{N-1}_+)\to L^q(\R^N)} \,,
$$
which is finite by the Stein--Tomas theorem. The operators $\mathcal V_{\delta^{-1}}$ and $\mathcal U_\delta$ are simply defined by
$$
\left( \mathcal V_{\delta} F\right)(x) = \delta^{-(N+1)/q} F(x'/\delta,x_N/\delta^2) \,,
\qquad
\left( \mathcal U_\delta \phi \right)(\xi) = \delta^{-(N-1)/2} \phi(\xi/\delta) \,. 
$$
The operator $\mathcal B$ is defined by
\begin{equation}\label{eq:defi-B0}
\left(\mathcal B\phi\right)\left(\frac{\xi}{\sqrt{1+\xi^2}}, \frac{1}{\sqrt{1+\xi^2}}\right)=(1+|\xi|^2)^{N/4}\phi(\xi) \,.
\end{equation}
The fact that $\mathcal B$ is a unitary operator from $L^2(\R^{N-1})$ to $L^2(\Sph^{N-1}_+)$ follows by a simple change of variables. The claimed identity \eqref{eq:tfactorization} follows by the same change of variables.
\end{proof}

We now use this lemma to prove the proposition.

\begin{proof}[Proof of Proposition \ref{tconv}]
Because of Lemma \ref{tbounded} it suffices to prove the proposition for $\psi\in L^2(\R^{N-1})$ with $\hat\psi\in C_c^\infty(\R^{N-1})$. For such $\psi$ we shall show that for all $(x',x_N)\in\R^N$, 
\begin{equation}\label{eq:ae-convergence}
\lim_{\delta\to0} \left( \mathcal T_\delta\psi\right)(x',x_N)=\left(e^{ix_N\Delta/2}\psi\right)(x')\,, 
\end{equation}
\begin{equation}\label{eq:uniform-decay}
\left|\left(\mathcal T_\delta\psi\right)(x',x_N)\right|+\left|\left(e^{ix_N\Delta/2}\psi\right)(x')\right|\leq C |x|^{-(N-1)/2}
\end{equation}
for some constant $C>0$ independent of $\delta$ (but dependent of $\psi$). The limit \eqref{eq:ae-convergence} follows immediately from Lebesgue's dominated convergence theorem, since we have the correct limit under the integral and $\hat \psi\in L^1(\R^{N-1})$. Assume for the moment the decay estimate \eqref{eq:uniform-decay} and let us show that this implies the claimed $L^q$ convergence. We have for some $C'$ and all $\delta>0$
$$
\int_{|x|\geq R} \left( \left|\left(\mathcal T_\delta \psi\right)(x) \right|^q + \left|\left(e^{ix_N\Delta/2}\psi\right)(x')\right|^q \right) dx \leq \frac{C'}{R}\,.
$$
This can be made arbitrarily small, uniformly in $\delta>0$, by choosing $R>0$ large. Thus, it suffices to prove that for any fixed $R>0$
$$
\chi_{B_R}(x) \left(\mathcal T_\delta\psi \right)(x) \to \chi_{B_R}(x) \left( e^{ix_N\Delta/2}\psi \right)(x')
\ \text{in}\ L^q(\R^N) \,,
$$
where $B_R$ denotes the ball of radius $R>0$. This follows immediately by dominated convergence from the pointwise convergence \eqref{eq:ae-convergence} together with the uniform bound
$$
\left|\left(\mathcal T_\delta\psi\right)(x)\right| \leq (2\pi)^{-(N-1)/2} \int_{\R^{N-1}} |\hat\psi(\xi)|\,d\xi <\infty \,.
$$

Thus, it thus remains to prove the decay estimate \eqref{eq:uniform-decay}, which follows from stationary phase estimates as in Stein \cite[p.349]{Stein-book-93}. Let us recall how it is done when there is no dependence on $\delta$. The integral
$$
\left(e^{ix_N\Delta/2}\psi\right)(x')=\frac{1}{(2\pi)^{(N-1)/2}}\int_{\R^{N-1}}e^{ix'\cdot\xi-ix_N \xi^2/2}\hat \psi(\xi)\,d\xi
$$
can be written as an oscillatory integral
$$\int_{\R^{N-1}}e^{i\lambda\Phi(\omega,\xi)}a(\xi)\,d\xi$$
with a large parameter $\lambda=|x|$, a smooth phase function $\Phi(\omega,\xi)=\omega'\cdot\xi-\omega_N\xi^2/2$, where $\omega=(\omega',\omega_N)\in\Sph^{N-1}$, and an amplitude $a=\hat\psi/(2\pi)^{N-1} \in C^\ii_c(\R^{N-1})$. We distinguish two cases: when $\omega$ is close to the poles, then the phase has critical points but we have a uniform lower bound on the determinant of the Hessian, so we may use stationary phase. Away from the poles, there is no critical point and we have a uniform lower bound on $|\nabla \Phi|$, so that we may use integration by parts.
  
First, when $|\omega_N|\ge b$ for some $0<b<1$ to be determined later (that is, when $\omega$ is close to the poles), then the phase is stationary at the points where
$$\nabla_\xi\Phi(\omega,\xi)=\omega'-\omega_N\xi=0,$$
that is for $\xi=\omega'/\omega_N$. Furthermore, we have $D^2_\xi\Phi(\omega,\xi)=\omega_N$, meaning that 
$$|\det D^2_\xi\Phi(\omega,\xi)|=|\omega_N|^{N-1}\ge b^{N-1}.$$
All the $\xi$-derivatives of $a$ and $\Phi$ are uniformly bounded in $\omega$ in this region, so that we may use the uniform stationary phase estimates of Alazard, Burq, and Zuily \cite{AlaBurZui-15} to infer that 
$$\left|\int_{\R^{N-1}}e^{i\lambda\Phi(\omega,\xi)}a(\xi)\,d\xi\right|\le C_{a,b}\lambda^{-(N-1)/2}$$
for all $\omega$ such that $|\omega_N|\ge b$. In the region $|\omega_N|\le b$, we have $|\omega'|\ge(1-b^2)^{1/2}$ and hence
$$|\nabla_\xi\Phi(\omega,\xi)|\ge\sqrt{1-b^2}-bR,$$
where $R>0$ is such that $\text{supp}\ a\subset B(0,R)$. Hence, if $b$ is sufficiently small such that $\sqrt{1-b^2}/b>R$, then the phase has no critical point and we have by integration by parts
$$\left|\int_{\R^{N-1}}e^{i\lambda\Phi(\omega,\xi)}a(\xi)\,d\xi\right|\le C_n(a)\lambda^{-n}$$
for any $n\in\N$, where $C_n(a)$ is uniform in $\omega$ such that $|\omega_N|\le b$, since we have a uniform lower bound on $|\nabla_\xi \Phi|$ in this region and uniform upper bounds on higher $\xi$-derivatives of $\Phi$. 

We have to do the same thing when $\delta>0$, and all the bounds that were uniform in $\omega$ should now be uniform in $(\omega,\delta)$. In this case, the new phase function has the form
$$
\Phi(\delta,\omega,\xi)=\omega'\cdot\xi\zeta_1(\delta|\xi|)-\omega_N\frac{\xi^2}{2}\zeta_2(\delta|\xi|)\,,
$$
and the amplitude has the form
$$
a(\delta,\xi)=\frac{\hat\psi(\xi)}{(1+\delta^2\xi^2)^{N/4}} \,.
$$
The functions $\xi\mapsto \zeta_1(\delta |\xi|)$ and $\xi\mapsto\zeta_2(\delta|\xi|)$, and $a$ are $C^\ii$ and satisfy $\zeta_1(0)=1=\zeta_2(0)$. All the $\xi$-derivatives of $\Phi$ and $a$ are uniformly bounded in $(\delta,\omega)$, on the support of $a$. First, consider the case $|\omega_N|\ge b$. We have 
$$D^2_\xi\Phi(\delta,\omega,\xi)=\omega_N \zeta_2(\delta|\xi|) +O(\delta),$$
where the $O(\delta)$ is uniform in $(\omega,\xi)\in\Sph^{N-1}\times\text{supp}(a)$. Hence, there exists $\delta_0=\delta_0(R)>0$ such that 
$$|\det D^2_\xi\Phi(\delta,\omega,\xi)|\ge\left(\frac b2\right)^{N-1},$$
for all $|\omega_N|\ge b$, $0\le\delta\le\delta_0$, $\xi\in\text{supp}(a)$.
Using again the result of Alazard--Burq--Zuily (notice here that we do not need to describe where the critical points are, a lower bound on the determinant of the Hessian is enough to apply their result - we could have done the same in the $\delta=0$ case actually), we obtain again that
$$\left|\int_{\R^{N-1}}e^{i\lambda\Phi(\delta,\omega,\xi)}a(\delta,\xi)\,d\xi\right|\le C_{a,b}\lambda^{-(N-1)/2}$$
for all $\omega$ such that $|\omega_N|\ge b$ and all $0\le\delta\le\delta_0$. In the region $|\omega_N|\le b$, we use the fact that
$$
\nabla_\xi\Phi(\delta,\omega,\xi)=\omega'\zeta_1(\delta|\xi|)-\omega_N\xi\zeta_2(\delta|\xi|)+O(\delta) \,,
$$
and hence
$$|\nabla_\xi\Phi(\delta,\omega,\xi)|\ge\frac12\sqrt{1-b^2}-\frac32 bR$$
for all $0\le\delta<\delta_1(b,R)$, $|\omega_N|\le b$, $\xi\in\text{supp}(a)$. For $b=b(R)$ small enough, this lower bound is positive. Using again integration by parts with this uniform lower bound on $|\nabla\Phi|$, we deduce
$$\left|\int_{\R^{N-1}}e^{i\lambda\Phi(\omega,\xi)}a(\delta,\xi)\,d\xi\right|\le C_n(a)\lambda^{-n}$$
for any $n\in\N$, where $C_n(a)$ is uniform in $(\omega,\delta)$ such that $|\omega_N|\le b$ and $0\le\delta\le\delta_1$. This finishes the proof of \eqref{eq:uniform-decay} and the proof of Proposition \ref{tconv}.
\end{proof}


\subsection{Almost everywhere convergence}

While in the previous subsection we dealt with $L^q$ convergence of $\mathcal T_{\delta_n}\psi$ when $\delta_n\to 0$, we will now deal with almost everywhere convergence of $\mathcal T_{\delta_n}\psi_n$ when $\delta_n\to 0$ and $\psi_n \rightharpoonup 0$ in $L^2(\R^{N-1})$. The purpose of this subsection is to prove the following convergence result.

\begin{proposition}\label{aeconv}
Let $\psi_n\rightharpoonup 0$ in $L^2(\R^{N-1})$ and $\delta_n\to 0$ in $(0,\infty)$. Then $\mathcal T_{\delta_n}\psi_n \to 0$ in $L^2_{\loc}(\R^N)$ and, in particular, there is a subsequence such that $\mathcal T_{\delta_{n_k}}\psi_{n_k} \to 0$ almost everywhere on $\R^N$.
\end{proposition}

The key ingredient in the proof of this proposition is the following analogue of the local smoothing property of the Schr\"odinger equation.

\begin{lemma}\label{locsmooth}
Let $a\in\mathcal S(\R^{N-1})$ be radial. Then there is a constant $C_a$ such that for all $\psi\in L^2(\R^{N-1})$ and all $\delta>0$
$$
\int_{\R^N} a(x') \left| \mathcal T_\delta \left(\frac{-\Delta}{-\delta^2\Delta +1}\right)^{1/4}\psi \right|^2 \,dx \leq C_a \|\psi\|^2 \,.
$$
\end{lemma}

Let us show that this lemma implies the proposition.

\begin{proof}[Proof of Proposition \ref{aeconv}]
Let $K\subset\R^N$ be compact and let $a\in\mathcal S(\R^{N-1})$ be radial with $\inf_{x\in K} a(x')>0$ (for instance a Gaussian). Moreover, let $\Lambda>0$ and denote by $P_\Lambda$ multiplication by the characteristic function of $B_\Lambda$, the ball of radius $\Lambda$, in Fourier space. We decompose, with $P_\Lambda^\bot =1-P_\Lambda$,
$$
\chi_K \mathcal T_{\delta_n}\psi_n = \chi_K \mathcal T_{\delta_n}P_\Lambda \psi_n + \chi_K \mathcal T_{\delta_n}P_\Lambda^\bot \psi_n \,.
$$
According to Lemma \ref{locsmooth} we have
\begin{align*}
\left\| \chi_K \mathcal T_{\delta_n}P_\Lambda^\bot \psi_n \right\| 
& \leq \left( \inf_{x\in K} a(x') \right)^{-1} \left\| a \mathcal T_\delta \left(\frac{-\Delta}{-\delta_n^2\Delta +1} \right)^{1/4} \right\| \left\| \left(\frac{-\delta_n^2\Delta +1}{-\Delta} \right)^{1/4} P_\Lambda^\perp \right\| \|\psi_n\| \\
& \leq \left( \inf_{x\in K} a(x') \right)^{-1} C_{a^2}^{1/2} \left( \delta_n^2 + \Lambda^{-2} \right)^{1/4} \sup_n \|\psi_n\| \,.
\end{align*}
The right side can be made arbitrarily small by choosing $\Lambda$ large, uniformly for large $n$. Therefore it suffices to prove that $\chi_K \mathcal T_{\delta_n}P_\Lambda \psi_n$ tends to zero for each fixed $\Lambda$. We will deduce this using dominated convergence. In fact, we have for each fixed $x\in\R^N$,
$$
\chi_{B_\Lambda}(\xi) e^{i\left( \xi\cdot x' \zeta_1(\delta|\xi|) - \frac12 \xi^2 x_N \zeta_2(\delta|\xi|)\right)} (1+\delta^2\xi^2)^{-N/4} \to \chi_{B_\Lambda}(\xi) e^{i\left( \xi\cdot x' - \frac12 \xi^2 x_N \right)}
$$
\emph{strongly} in $L^2_\xi(\R^{N-1})$. (This can also be proved with the help of dominated convergence.) Thus, $\hat\psi_n\rightharpoonup 0$ implies that for any fixed $x\in\R^N$,
$$
\mathcal T_{\delta_n}P_\Lambda \psi_n (x) = \left( \chi_{B_\Lambda} e^{-i\left( \xi\cdot x' \zeta_1(\delta|\xi|) - \frac12 \xi^2 x_N \zeta_2(\delta|\xi|)\right)} (1+\delta^2\xi^2)^{-N/4}, \hat\psi_n \right) \to 0 \,.
$$
Moreover, we have
\begin{align*}
\left| \mathcal T_{\delta_n}P_\Lambda \psi_n (x) \right| & \leq \left\| \chi_{B_\Lambda} e^{-i\left( \xi\cdot x' \zeta_1(\delta|\xi|) - \frac12 \xi^2 x_N \zeta_2(\delta|\xi|)\right)} (1+\delta^2\xi^2)^{-N/4} \right\| \left\| \hat\psi_n \right\| \\
& \leq |B_\Lambda|^{1/2} \sup_n \|\psi_n\| \,.
\end{align*}
Thus, dominated convergence implies $\chi_K \mathcal T_{\delta_n}P_\Lambda \psi_n\to 0$ in $L^2(\R^N)$, which proves the first part of the proposition.

The second part follows by a standard diagonalization argument using a sequence of balls with diverging radii and the fact that an $L^1$ convergent sequence has an almost everywhere convergent subsequence.
\end{proof}

It remains to give the

\begin{proof}[Proof of Lemma \ref{locsmooth}]
Expanding the square, the left side of the term in the lemma reads
\begin{multline*}
2\pi \iint_{\R^{N-1}\times\R^{N-1}} \overline{\hat{\psi}(\xi)}\frac{|\xi|^{1/2}}{(1+\delta^2\xi^2)^{(N+1)/4}}\hat{a}(\xi \zeta_1(\delta|\xi|)-\xi' \zeta_1(\delta|\xi'|))\times\\
\times\delta\left(\frac{\xi^2}{2}\zeta_2(\delta|\xi|)-\frac{\xi'^2}{2}\zeta_2(\delta|\xi'|)\right)\frac{|\xi'|^{1/2}}{(1+\delta^2\xi'^2)^{(N+1)/4}} \hat{\psi}(\xi')\,d\xi\,d\xi'.
\end{multline*}
By the Schur test for boundedness, the lemma will follow if we can bound
\begin{align*}
\sup_{\xi}\int_{\R^{N-1}}\frac{|\xi|^{1/2}|\xi'|^{1/2}}{(1+\delta^2 \xi^2)^{(N+1)/4}(1+\delta^2 \xi'^2)^{(N+1)/4}} & |\hat{a}(\xi \zeta_1(\delta|\xi|)-\xi' \zeta_1(\delta|\xi'|))| \times \\
& \times \delta\left(\frac{\xi^2}{2}\zeta_2(\delta|\xi|)-\frac{\xi'^2}{2}\zeta_2(\delta|\xi'|)\right)\,d\xi'
\end{align*}
independently of $\delta$. In order to perform the $\xi'$ integral we write $\xi'=k\omega$ with $\omega\in\Sph^{N-2}$ and $k>0$. The functions $\Phi_\delta(k)= k^2\zeta_2(\delta k)/2 =\delta^{-2}\left( 1 - 1/\sqrt{1+\delta^2 k^2}\right)$ is a strictly increasing function, so we can change variables $\kappa = \Phi_\delta(k)$. We use the fact that
$$
\int_0^\infty F(k)\, \delta(\Phi_\delta(|\xi|)-\Phi_\delta(k))\,dk=\int_0^{\delta^{-2}} F(\Phi_\delta^{-1}(\kappa))\, \delta(\Phi_\delta(|\xi|)-\kappa)\frac{d\kappa}{|\Phi_\delta'(\Phi_\delta^{-1}(\kappa))|}=\frac{F(|\xi|)}{|\Phi_\delta'(|\xi|)|}
$$
with $\Phi_\delta'(k) = k(1+\delta^2 k^2)^{-3/2}$. So Schur's test amounts to estimating 
$$
\sup_\xi\frac{|\xi|^{N-2}}{(1+\delta^2\xi^2)^{(N-2)/2}} \int_{\Sph^{N-2}}|\hat{a}((\xi -|\xi|\omega)\zeta_1(\delta|\xi|))|\,d\omega \,.
$$
When $N=2$, this is equal to
$$
\sup_\xi \left( |\hat{a}(0)| + |\hat{a}(2\xi \zeta_1(\delta|\xi|))| \right),
$$
which is bounded since $\hat a$ is bounded.

In the remainder of the proof we assume $N\geq 3$. Since $a$ is assumed to be radial, by rotation invariance we may choose $\xi=|\xi|(0,\ldots,0,1)$ and then the integral over the sphere becomes
$$
|\Sph^{N-3}| \int_0^\pi|\hat{a}(2\sin(\theta/2) |\xi| \zeta_1(\delta|\xi|)) | (\sin\theta)^{N-3}\,d\theta \,.
$$
We distinguish between two regions: when $2\sin(\theta/2) |\xi| \zeta_1(\delta|\xi|)\geq 1$, then we estimate $|\hat{a}(\xi')|\le c_n|\xi'|^{-n}$ for any $n\in\N$ and obtain 
\begin{multline*}
  \int_{2\sin(\theta/2) |\xi| \zeta_1(\delta|\xi|)\geq 1} |\hat{a}(2\sin(\theta/2) |\xi| \zeta_1(\delta|\xi|)) | (\sin\theta)^{N-3}\,d\theta \\
  \le c_n (2|\xi|\zeta_1(\delta|\xi|))^{-n} \int_{\sin(\theta/2)\geq(2|\xi|\zeta_1(\delta|\xi|))^{-1}}\frac{(\sin\theta)^{N-3}}{(\sin(\theta/2))^n}\,d\theta,
\end{multline*}
Fix $n\geq N-4$. Then for large $|\xi|$, the integral blows up as 
\begin{align*}
\int_{\sin(\theta/2)\geq(2|\xi|\zeta_1(\delta|\xi|))^{-1}} \frac{(\sin\theta)^{N-3}}{(\sin(\theta/2))^n}\,d\theta& \sim 2^n \int_{\theta \geq(|\xi|\zeta_1(\delta|\xi|))^{-1}} \theta^{N-3-n}\,d\theta \\
& = \frac{2^n}{n-N+3} (|\xi| \zeta_1(\delta|\xi|))^{-(N-2-n)}.
\end{align*}
Thus, we find that 
$$
\int_{2\sin(\theta/2) |\xi| \zeta_1(\delta|\xi|)\geq 1} |\hat{a}(2\sin(\theta/2) |\xi| \zeta_1(\delta|\xi|)) | (\sin\theta)^{N-3}\,d\theta
\lesssim_n (|\xi| \zeta_1(\delta|\xi|))^{-(N-2)}.
$$
In the region $2\sin(\theta/2) |\xi| \zeta_1(\delta|\xi|)\leq 1$, we estimate $|\hat{a}|\leq c$ and obtain 
\begin{multline*}
  \int_{2\sin(\theta/2) |\xi| \zeta_1(\delta|\xi|)\leq 1} |\hat{a}(2\sin(\theta/2) |\xi| \zeta_1(\delta|\xi|)) | (\sin\theta)^{N-3}\,d\theta \\
  \le c \int_{\theta\leq(|\xi|\zeta_1(\delta|\xi|))^{-1}} |\theta|^{N-3}\,d\theta
  \lesssim (|\xi| \zeta_1(\delta|\xi|))^{-(N-2)}.
\end{multline*}

Inserting this bound into the supremum in Schur's test, we obtain
$$
\sup_\xi\frac{1}{(1+\delta^2\xi^2)^{(N-2)/2}}\frac{1}{\zeta_1(\delta|\xi|)^{N-2}} = 1 \,.
$$
This proves the lemma.
\end{proof}


\section{Compactness}\label{sec:comp}

In this section we prove a refinement of the Stein--Tomas inequality and deduce a compactness theorem modulo modulations and concentrations. We recall that the convergence $f_n\rightharpoonup_{\text{conc}}$ was introduced in terms of the operators $\mathcal B_{R,\delta}$ with $R\in\cO(N)$ and $\delta>0$ which identify pairs of functions on $L^2(\R^{N-1})$ with functions on $L^2(\Sph^{N-1})$. The parameter $R\in\cO(N)$ determines the equator along which we cut the function in $L^2(\Sph^{N-1})$ into two pieces. The parameter $\delta>0$ corresponds to a scaling in $L^2(\R^{N-1})$. The precise definition of these operators is given in \eqref{eq:defi-B}. The refined Stein--Tomas inequality is stated in Subsection \ref{sec:steintomref}, where we also use it to deduce the compactness theorem, and is proved in Subsection \ref{sec:steintomrefproof} (see also Appendix \ref{sec:strichref}).


\subsection{Refinement of the Stein--Tomas inequality}\label{sec:steintomref}

Our refined Stein--Tomas inequality depends on a parameter $\epsilon\in (0,1)$ that will be chosen small enough and that will not always be reflected in the notation. Given this parameter we consider for any $\theta\in\Sph^{N-1}$ the cap
$$
C(\theta) := \left\{ \omega\in\Sph^{N-1}:\ \theta\cdot\omega > \sqrt{1-\epsilon^2} \right\} \,,
$$
and we also pick an orthogonal matrix $R_\theta\in\cO(N)$ mapping the north pole to $\theta$:
$$R_\theta(0,\ldots,0,1)=\theta.$$
To formulate our refinement of the Stein--Tomas inequality we need an analogue of dyadic cubes on the sphere. Let $\mathcal D$ denotes the set of all dyadic cubes in $\R^{N-1}$, that is, the union over $j\in\Z$ of all cubes of side length $2^j$ with corners on $(2^j\Z)^{N-1}$. For $\theta\in\Sph^{N-1}$ and $Q\in\mathcal D$ we let
$$
L_\theta(Q) := R_\theta(L(Q)) \,,
$$
where $L$ stands for ``lift'' and
$$
L(Q) := \{ \omega\in\Sph^{N-1}:\ \omega'\in Q \,,\ \omega_N> 0 \}.
$$
Notice that 
$$L_\theta(Q)=\{ \omega\in\Sph^{N-1}:\ P_{\theta^\perp}(\omega)\in Q\,,\ \omega\cdot\theta>0\},$$
where $P_{\theta^\perp}:\R^N\to\R^N$ is the orthogonal projection on $\theta^\perp$. 
By compactness of the sphere, there is a finite $A\in\N$ and points $\theta_\alpha\in\Sph^{N-1}$, $\alpha=1,\ldots,A$, such that
$$
\bigcup_{\alpha=1}^A C(\theta_\alpha) = \Sph^{N-1} \,.
$$
Correspondingly, we choose non-negative continuous functions $\chi_\alpha$, $\alpha=1,\ldots,A$, with
$$
\sum_{\alpha=1}^A \chi_\alpha = 1
\qquad\text{and}\qquad
\supp\chi_\alpha \subset C(\theta_\alpha) 
\ \text{for all}\ \alpha =1,\ldots,A \,.
$$

\begin{proposition}\label{steintomref}
There are $\epsilon\in(0,1)$, $C>0$ and $\sigma\in (0,1)$ such that for any $f\in L^2(\Sph^{N-1})$,
\begin{equation}
\label{eq:steintomref}
\| \check f\|_{L^q(\R^N)} \leq C \left( \sup_\alpha \sup_{Q\in\mathcal D} |Q|^{-1/2} \left\| \left(\1_{L_{\theta_\alpha}(Q)} \chi_\alpha f \right)^\vee \right\|_{L^\infty(\R^N)} \right)^{1-\sigma} \|f\|^\sigma \,.
\end{equation}
\end{proposition}

Note that this proposition implies the standard Stein--Tomas inequality: indeed, for any $\alpha\in\{1,\ldots,A\}$ and $Q\in\mathcal D$ we have
\begin{align*}
|Q|^{-1/2} \left\| \left(\1_{L_{\theta_\alpha}(Q)} \chi_\alpha f \right)^\vee \right\|_{L^\infty(\R^N)}
& \leq (2\pi)^{-N/2} |Q|^{-1/2} \left\|\1_{L_{\theta_\alpha}(Q)} \chi_\alpha f \right\|_{L^1(\Sph^{N-1})}\\
& \leq (2\pi)^{-N/2} |Q|^{-1/2} \left\|\1_{L_{\theta_\alpha}(Q)} \chi_\alpha \right\|_{L^2(\Sph^{N-1})} \left\| f \right\|_{L^2(\Sph^{N-1})}\\
& \leq C_N \left\| f \right\|_{L^2(\Sph^{N-1})} \,,
\end{align*}
so the right side of \eqref{eq:steintomref} is bounded by a constant times $\|f\|$.

Other refinements of the Stein--Tomas inequality can be found in \cite[Thm. 4.2]{MoyVarVeg-99} in the case $N=3$ or in \cite[Prop. 2]{Oliveira-14}, \cite[Prop. 4.1]{Shao-15} in the case $N=2$. These refinements involve $X_{p}$-norms on $f$, but it is not obvious how to deduce our compactness result (Corollary \ref{concentration}) from these $X_{p}$ estimates. The key feature of our refinement is the $L^\infty$ norm of the Fourier transform on the right side of \eqref{eq:steintomref}, leads almost immediately to Corollary \ref{concentration}. This is reminiscent of the route taken in \cite{Tao-09,KilVis-book} in connection with the Strichartz inequality, where also $L^\ii$ bounds on the Fourier transform are used instead of the original $X_p$-spaces approach of \cite{Bourgain-98,MerVeg-98,CarKer-07,BegVar-07}.

We also point out a certain similarity with the description of lack of compactness in homogeneous Sobolev spaces \cite{Ge}. In this case analogous bounds in terms of $L^\infty$ norms of the Fourier transform are due to \cite{GeMeOr} (see also \cite[Prop. 4.8]{KilVis-book}) and have been used to establish compactness results \cite{Ge} (see also \cite[Prop. 4.9]{KilVis-book}).

We defer the proof of Proposition \ref{steintomref} to Subsection \ref{sec:steintomrefproof} and use it now to deduce our compactness theorem. The relation between our convergence notion $f_n\rightharpoonup_{\text{conc}} 0$ and the norm appearing in Proposition \ref{steintomref} is clarified in the following lemma.

\begin{lemma}\label{lem:Linfty-estimate-Tao}
The following holds provided $\epsilon>0$ is sufficiently small. If $(f_n)$ is a bounded sequence in $L^2(\Sph^{N-1})$ with $f_n\rightharpoonup_{\text{conc}}0$, then
  \begin{equation}\label{eq:Linfty-estimate-Tao}
    \lim_{n\to\ii} \sup_\alpha \sup_{Q\in\mathcal D} |Q|^{-1/2} \left\| \left(\1_{L_{\theta_\alpha}(Q)} \chi_\alpha f_n \right)^\vee \right\|_{L^\infty(\R^N)} = 0.
  \end{equation}
\end{lemma}

\begin{proof}
We argue by contradiction: assume that there exists $\epsilon'>0$, $\alpha\in\{1,\ldots,A\}$, sequences $(x_k)\subset\R^N$, $(Q_k)\subset\mathcal D$ and a subsequence $(f_{n_k})$ such that for all $k$
\begin{equation}
\label{eq:compest}
|Q_{k}|^{-1/2} \left|\left(\1_{L_{\theta_\alpha}(Q_{k})} \chi_\alpha f_{n_k} \right)^\vee(x_{k})\right|\ge\epsilon'.
\end{equation}
We show that the left side converges to zero, obtaining the desired contradiction. In the sequel, we forget about the subsequence and write $(f_n)_n$ instead of $(f_{n_k})_k$. We may also assume that $\theta_\alpha=(0,\ldots,0,1)$ up to replacing $f_n$ by $f_n\circ R_{\theta_\alpha}$, which does not change the assumption $f_n\rightharpoonup_{\text{conc}}0$. We thus write $L(Q)$ and $\chi$ instead of $L_{\theta_\alpha}(Q)$ and $\chi_\alpha$. We may assume that the sets $L(Q_n)$ intersect $\{\omega\subset\Sph^{N-1}:\ \omega_N > \sqrt{1-\epsilon^2}\}$ (which contains the support of $\chi$), for otherwise the left side of \eqref{eq:compest} vanishes, and therefore the cubes $Q_n$ all intersect $\{\xi\subset\R^{N-1}:\ |\xi|<\epsilon\}$. Let $\tilde Q_n$ be the smallest dyadic cube with $\1_{L(\tilde Q_n)} \chi = \1_{L(Q_n)} \chi$. Since
$$
|\tilde Q_n|^{-1/2}\left| \left(\1_{L(\tilde Q_n)} \chi f_n \right)^\vee(x_n) \right| \geq |Q_n|^{-1/2}\left| \left(\1_{L(Q_n)} \chi f_n \right)^\vee(x_n) \right| \,,
$$
it suffices to prove the convergence to zero with $\tilde Q_n$ in place of $Q_n$. From now on we will write again $Q_n$ instead of $\tilde Q_n$. Let $k_n\in\Z^{N-1}$ and $\delta_n\in2^\Z$ such that
$$
Q_n=\delta_nk_n+[0,\delta_n)^{N-1} \,,
$$
and note that $|Q_n|=\delta_n^{N-1}$. The above redefinition of $Q_n$ guarantees that the sequence $(\delta_n k_n)\subset\R^{N-1}$ belongs to a compact set (of diameter $\cO(\epsilon)$) and that the sequence $(\delta_n)\subset (0,\infty)$ is bounded (by $\cO(\epsilon)$). Thus, after passing to a subsequence if necessary, we may assume that $(\delta_nk_n)$ and $(\delta_n)$ converge.

For any $\theta\in C(0,\ldots,0,1)$, we define a rotation $O_\theta\in\cO(N)$ that sends $(0,\ldots,0,1)$ to $\theta$ in the following fashion: if $\theta=(0,\ldots,0,1)$, we take $O_\theta=\text{Id}$, and if $\theta\neq(0,\ldots,0,1)$, we take $O_\theta=\text{Id}$ on the orthogonal complement of $H=\text{span}((0,\ldots,0,1),\theta)$ and on $H$ we take 
$$O_\theta=\left(\begin{array}{cc}
                  \omega_N & - |\omega'| \\
                  |\omega'| & \omega_N
                 \end{array}\right)
$$
in the orthonormal basis $((0,\ldots,0,1),\omega'/|\omega'|)$ of $H$ (with the notation $\omega=\omega'+\omega_N(0,\ldots,0,1)$, $\omega'\in\R^{N-1}\times\{0\}$). This definition ensures that $\theta\mapsto O_\theta$ is continuous on $C(0,\ldots,0,1)$. Next, we define
$$
\theta_n:=(\delta_nk_n,\sqrt{1-\delta_n^2|k_n|^2})\in\Sph^{N-1}
$$
and 
$$
(\phi_n^+,\phi_n^-) := \mathcal B^{-1}_{O_{\theta_n},\delta_n}(e^{ix_n\cdot\omega}f_n) \,.
$$
By choosing $\epsilon>0$ small enough (depending only on $N$) we can guarantee that $\theta_n\cdot\omega>0$ for all $n$ and all $\omega\in\Sph^{N-1}$ with $\omega_N>\sqrt{1-\epsilon^2}$. We conclude that
\begin{align*}
\left(\1_{L(Q_n)} \chi f_n \right)^\vee(x_n)
& = (2\pi)^{-N/2} \int_{\Sph^{N-1}}e^{ix_n\cdot\omega}f_n(\omega)\chi(\omega)\1_{Q_n}(\omega') \,d\omega \\
& = (2\pi)^{-N/2} \int_{\R^{N-1}} \phi_n^+(\xi)h_n(\xi)\,d\xi
\end{align*}
with
\begin{multline*}
h_n(\xi):=(1+\delta_n^2|\xi|^2)^{-N/4} \, \chi\left(\frac{\theta_n+O_{\theta_n}(\delta_n\xi,0)}{\sqrt{1+\delta_n^2|\xi|^2}}\right) \, \1_{Q_n}\left( P\left(\frac{\theta_n+O_{\theta_n}(\delta_n\xi,0)}{\sqrt{1+\delta_n^2|\xi|^2}}\right)\right).
\end{multline*}
and with the projection $P:\R^N\to\R^{N-1}$ defined by $P(\eta',\eta_N):=\eta'$. 

Since $\phi_n^+\rightharpoonup0$ in $L^2(\R^{N-1})$ by assumption, our claim \eqref{eq:compest} will follow if we can prove that $(h_n)$ converges strongly in $L^2(\R^{N-1})$. To do so, we prove that $(h_n)$ converges almost everywhere and that $0\le h_n\le \1_B$ for a centered ball $B$ with (finite) radius independent of $n$.

We begin with the almost everywhere convergence. Since $(\theta_n)$ and $(\delta_n)$ converge and $\chi$ and $\theta\mapsto O_\theta$ are continuous, the sequence
$$
(1+\delta_n^2|\xi|^2)^{-N/4} \, \chi\left(\frac{\theta_n+O_{\theta_n}(\delta_n\xi,0)}{\sqrt{1+\delta_n^2|\xi|^2}}\right)
$$
converges for all $\xi$. If the limit of $(\delta_n)$ is positive, then the cube $Q_n$ converges towards a fixed cube, and thus the sequence 
$$
\1_{Q_n} \left(P\left(\frac{\theta_n+O_{\theta_n}(\delta_n\xi,0)}{\sqrt{1+\delta_n^2|\xi|^2}}\right)\right)
$$
converges almost everywhere in $\xi$. If $\lim_n \delta_n=0$, then we use the fact that
$$
P\left(\frac{\theta_n+O_{\theta_n}(\delta_n\xi,0)}{\sqrt{1+\delta_n^2|\xi|^2}}\right)\in Q_n
$$ 
if and only if 
\begin{equation}
\label{eq:constraintq}
P(O_{\theta_n}(\xi,0))\in\left(1-\frac{1}{\sqrt{1+\delta_n^2|\xi|^2}}\right)k_n+[0,1)^{N-1} \,.
\end{equation}
Since 
$$
\lim_{n\to\ii}\left(1-\frac{1}{\sqrt{1+\delta_n^2|\xi|^2}}\right)k_n=0 \,,
$$
 we also have almost everywhere convergence in the case $\lim_n \delta_n=0$.
 
Let us now show that $0\le h_n\le\1_B$ for a centered ball $B$ with (finite) radius independent of $n$. Since $O_\theta\to\text{Id}$ as $\theta\to(0,\ldots,0,1)$ and $|\theta_n-(0,\ldots,0,1)|=\cO(\epsilon)$ uniformly in $n$, we choose $\epsilon>0$ small enough such that for all $\xi\in\R^{N-1}$ and all $n$,
$$
|P O_{\theta_n}(\xi,0)|\geq\frac12|\xi| \,.
$$
Now assume that $\xi\in\supp h_n$. Then \eqref{eq:constraintq} and the fact that $1-(1+x)^{-1/2}\le \min\{1,x\}$ for all $x\geq 0$ implies that
\begin{align*}
\frac12|\xi| & \le|PO_{\theta_n}(\xi,0)|\le \left(1-\frac{1}{\sqrt{1+\delta_n^2|\xi|^2}}\right)|k_n|+\cO(1)\le \min\{1,\delta_n^2|\xi|^2\} \, |k_n| + \cO(1) \\
& = \min\{\delta_n^{-1}|\xi|^{-1},\delta_n|\xi|\} \, \delta_n |k_n||\xi| + \cO(1)
\le \delta_n |k_n||\xi| + \cO(1) \,,
\end{align*}
Recalling that $\delta_n|k_n|\leq C \epsilon$ and choosing $\epsilon<1/(2C)$, we conclude that $|\xi|=\cO(1)$ uniformly in $n$, which is what we want to prove. This concludes the proof of Lemma \ref{lem:Linfty-estimate-Tao}.
\end{proof}

Combining Proposition \ref{steintomref} with Lemma \ref{lem:Linfty-estimate-Tao} we obtain immediately the following compactness result.

\begin{corollary}\label{concentration}
Let $(f_n)\subset L^2(\Sph^{N-1})$ with $\|f_n\|=1$ satisfy $f_n\rightharpoonup_{\text{conc}} 0$. Then $\check f_n\to 0$ in $L^q(\R^N)$.
\end{corollary}



\subsection{Proof of Proposition \ref{steintomref}}\label{sec:steintomrefproof}

Our goal in this subsection is to prove the refined Stein--Tomas inequality \eqref{eq:steintomref}. We will deduce this inequality from a refinement of a `perturbed Strichartz inequality', which we state next. We use the notation
$$
T(E) := 1-\sqrt{1-E}
\qquad\text{for}\ 0\leq E\leq 1 \,.
$$
and define $\psi_Q$ by $\hat\psi_Q = \chi_Q \hat\psi$ for $Q\in\mathcal D$, the collection of all dyadic cubes. Moreover, it is more natural to write $d$ instead of $N-1$, so that $q=2+4/d$.

\begin{proposition}\label{stref}
There are $\epsilon\in(0,1)$, $C>0$ and $\sigma\in (0,1)$ such that for any $\psi\in L^2(\R^d)$ with $\supp\hat\psi\subset\{|\xi|\leq\epsilon\}$,
\begin{equation}
\label{eq:stref}
\|e^{-itT(-\Delta)}\psi\|_{L^q_{t,x}} \leq C \left( \sup_{Q\in\mathcal D} |Q|^{-1/2} \|e^{-itT(-\Delta)}\psi_Q\|_{L^\infty_{t,x}} \right)^{1-\sigma} \|\psi\|_{L^2_x}^\sigma \,.
\end{equation}
\end{proposition}

This should be viewed as a perturbed Strichartz inequality since $T(\xi^2)\sim \xi^2/2$ as $\xi\to 0$. The analogue of Proposition \ref{stref} with $T(-\Delta)$ replaced by $-\Delta/2$ is essentially due to \cite{Tao-09} and appears in a slightly stronger form in \cite{KilVis-book}. (In this case the restriction on the support of $\hat\psi$ is not necessary.) Proposition \ref{stref} follows in the same way, but for the sake of completeness we provide the details in the appendix. As in \cite{Tao-09,KilVis-book} the crucial ingredient is Tao's bilinear restriction estimate \cite{Tao-03}.

With the refinement of the perturbed Strichartz inequality, Proposition \ref{stref}, at hand it is easy to give the

\begin{proof}[Proof of Proposition \ref{steintomref}]
We fix $\epsilon>0$ as given by Proposition \ref{stref}. Let $f\in L^2(\Sph^{N-1})$ have support in the cap $\{\omega:\ \omega_N> \sqrt{1-\epsilon^2}\}$ and define a function $\psi\in L^2(\R^{N-1})$ by
$$
\hat\psi(\xi) := \frac{f(\xi,\sqrt{1-\xi^2})}{\sqrt{1-\xi^2}} \,,
$$
so that
$$
\check f(x) = (2\pi)^{-1/2} e^{ix_N} \left(e^{-ix_N T(-\Delta)}\psi \right)(x')
$$
Since $\supp\hat\psi\subset\{|\xi|>\epsilon\}$ we can apply Proposition \ref{stref} and obtain
$$
\|\check{f}\|_{L^q(\R^N)} \leq (2\pi)^{-1/2} C \left( \sup_{Q\in\mathcal D} |Q|^{-1/2} \|e^{-itT(-\Delta)}\psi_Q\|_{L^\infty_{t,x}} \right)^{1-\sigma} \|\hat\psi\|_{L^2_\xi}^\sigma \,.
$$
We bound
$$
\|\hat\psi\|_{L^2_\xi}^2 \leq \frac{1}{\sqrt{1-\epsilon^2}} \|f\|^2
$$
and note that
$$
(\1_{L(Q)} f)^\vee(x) = (2\pi)^{-1/2} e^{ix_N} \left(e^{-ix_N T(-\Delta)}\psi_Q \right)(x') \,.
$$
Thus we conclude that
$$
\|\check f\|_{L^q(\R^N)} \leq C' \left( \sup_{Q\in\mathcal D} |Q|^{-1/2} \|(\1_{L(Q)} f)^\vee \|_{L^\infty(\R^N)} \right)^{1-\sigma} \|f\|_{L^2(\Sph^{N-1})}^\sigma \,.
$$
By rotation invariance of the sphere we obtain for $f\in L^2(\Sph^{N-1})$ with $\supp f\subset C(\theta_\alpha)$ the same inequality with $L(Q)$ replaced by $L_{\theta_\alpha}(Q)$.

Thus, for an arbitrary function $f\in L^2(\Sph^{N-1})$ we obtain
\begin{align*}
\|\check f\|_{L^q(\R^N)} & = \left\| \sum_{\alpha=1}^A (\chi_\alpha f)^\vee \right\|_{L^q(\R^N)} \leq \sum_{\alpha=1}^A \left\| (\chi_\alpha f)^\vee \right\|_{L^q(\R^N)} \\
& \leq C' \sum_{\alpha=1}^A \left( \sup_{Q\in\mathcal D} |Q|^{-1/2} \|(\1_{L_{\theta_\alpha}(Q)} \chi_\alpha f)^\vee \|_{L^\infty(\R^N)} \right)^{1-\sigma} \|\chi_\alpha f\|_{L^2(\Sph^{N-1})}^\sigma \\
& \leq C' \left( \sup_\alpha \sup_{Q\in\mathcal D} |Q|^{-1/2} \|(\1_{L_{\theta_\alpha}(Q)} \chi_\alpha f)^\vee \|_{L^\infty(\R^N)} \right)^{1-\sigma} \sum_{\alpha=1}^A \|\chi_\alpha f\|_{L^2(\Sph^{N-1})}^\sigma \\
& \leq C' \left( \sup_\alpha \sup_{Q\in\mathcal D} |Q|^{-1/2} \|(\1_{L_{\theta_\alpha}(Q)} \chi_\alpha f)^\vee \|_{L^\infty(\R^N)} \right)^{1-\sigma} A^{1-\sigma/2} \left( \sum_{\alpha=1}^A \|\chi_\alpha f\|_{L^2(\Sph^{N-1})}^2\right)^{\sigma/2} \\
& \leq C' \left( \sup_\alpha \sup_{Q\in\mathcal D} |Q|^{-1/2} \|(\1_{L_{\theta_\alpha}(Q)} \chi_\alpha f)^\vee \|_{L^\infty(\R^N)} \right)^{1-\sigma} A^{1-\sigma/2} \|f\|_{L^2(\Sph^{N-1})}^\sigma \,.
\end{align*}
This is the claimed inequality.
\end{proof}


\section{Equal profiles}\label{sec:equal}

Our goal in this section is to prove \eqref{eq:step3}, that is, we want to express the solution of the minimization problem $\tilde{\mathcal S}_d$ in terms of the solution of the minimization problem $\mathcal S_d$. This will follow from a general inequality that we describe next.

For $f,g\in L^q(\R^N)$ (in this section $q$ can be any number $\geq 2$) let
$$
\Phi_q(f,g) := \lim_{\lambda\to\infty} \int_{\R^N} \left| f(x) + e^{i\lambda x_N} g(x)\right|^q \,dx \,.
$$
It is easy to see that this limit exists and is given by
$$
\Phi_q(f,g) = \frac{1}{2\pi} \int_{\R^N} \int_{-\pi}^\pi \left| f(x) + e^{i\theta} g(x)\right|^q \,d\theta\,dx \,.
$$
A simple proof of this fact can be found, for instance, in \cite[Lem.~5.2]{Allaire-92}. Note that for fixed $x\in\R^N$, the function
$$
\theta\mapsto \left| f(x) + e^{i\theta} g(x)\right|^q = \left( |f(x)|^2 + |g(x)|^2 + 2 \re e^{i\theta} \overline{f(x)}g(x) \right)^{q/2}
$$
is continuous and has maximum $(|f(x)|+ |g(x)|)^q$. This maximum belongs to $L^1(\R^N)$ as function of $x$. Therefore, Allaire's result applies in the above setting.

\begin{lemma}\label{equal}
If $q\geq 2$, then
$$
\Phi_q(f,g) \leq \frac{2^{q/2}}{\sqrt\pi} \frac{\Gamma(\frac{q+1}{2})}{\Gamma(\frac{q+2}{2})} \left( \|f\|_q^2 + \|g\|_q^2 \right)^{q/2} \,.
$$
\end{lemma}

The importance of the constant on the right side is that we get equality if $|f|=|g|$. In fact, the proof below shows that if $q>2$, then the inequality is strict unless $|f|=|g|$ almost everywhere.

\begin{proof}
Let us write the above formula for $\Phi_q(f,g)$ as
$$
\Phi_q(f,g) = \int_{\R^N} \left( |f(x)|^2 + |g(x)|^2 \right)^{q/2} \phi(|\alpha(x)|) \,dx
$$
with
$$
\alpha(x) := \frac{2 \overline{f(x)} g(x)}{|f(x)|^2 + |g(x)|^2} \,.
$$
and, for $t\in [0,1]$,
$$
\phi(t) = \frac{1}{\pi} \int_0^\pi \left( 1 + t \cos\theta \right)^{q/2} \,d\theta \,.
$$
We claim that $\phi$ is increasing in $[0,1]$. In fact,
\begin{align*}
\phi'(t) & = \frac{q}{2} \frac{1}{\pi} \int_0^\pi \left( 1 + t \cos\theta \right)^{(q-2)/2}\cos\theta \,d\theta \\
& = \frac{q}{2} \frac{1}{\pi} \int_0^{\pi/2} \left( \left( 1 + t \cos\theta \right)^{(q-2)/2} - \left( 1 - t \cos\theta \right)^{(q-2)/2} \right) \cos\theta \,d\theta \,.
\end{align*}
For $q\geq 2$, the integrand on the right side is pointwise non-negative, which proves the monotonicity.

Since $|\alpha(x)|\leq 1$, we deduce that
$$
\Phi_q(f,g) \leq \phi(1) \int_{\R^N} \left( |f(x)|^2 + |g(x)|^2 \right)^{q/2} \,dx
$$
and therefore, by the triangle inequality in $L^{q/2}$,
\begin{align*}
\Phi_q(f,g)^{2/q} & \leq \phi(1)^{2/q} \left\| |f|^2 + |g|^2 \right\|_{q/2} \leq \phi(1)^{2/q} \left( \left\||f|^2\right\|_{q/2} + \left\||g|^2 \right\|_{q/2} \right) \\
& = \phi(1)^{2/q} \left( \left\|f\right\|^2_q + \left\|g \right\|_q^2 \right).
\end{align*}
Thus, to complete the proof of the lemma it remains to compute the value of $\phi(1)$. Using the integral representation of the beta function, we find
\begin{align}
\label{eq:gammafcn}
\phi(1) & = \frac1\pi \int_0^\pi \left( 1 + \cos\theta\right)^{q/2}\,d\theta = \frac{2^{q/2}}{\pi} \int_0^\pi \cos^q(\theta/2)\,d\theta = \frac{2^{(q+2)/2}}{\pi} \int_0^{\pi/2} \cos^q\theta\,d\theta \notag\\
& = \frac{2^{q/2}}{\pi} B(\frac12,\frac{q+1}2) = \frac{2^{q/2}}{\sqrt\pi} \frac{\Gamma(\frac{q+1}{2})}{\Gamma(\frac{q+2}{2})} \,.
\end{align}
This completes the proof.
\end{proof}

\begin{corollary}\label{step3}
$\tilde{\mathcal S}_d = \frac{2^{q/2}}{\sqrt\pi} \frac{\Gamma(\frac{q+1}{2})}{\Gamma(\frac{q+2}{2})}\ \mathcal S_d$ with $q=2+4/d$.
\end{corollary}

\begin{proof}
Let $\psi^+,\psi^-\in L^2(\R^d)$. By the lemma (with $N=d+1$) and the Strichartz inequality,
\begin{align*}
& \lim_{\lambda\to\infty} \iint_{\R\times\R^d} \left| e^{it\Delta/2}\psi^+(x) + e^{i\lambda x_N}e^{-it\Delta/2}\psi^-(x)\right|^q\,dx\,dt \\
& \qquad \leq \frac{2^{q/2}}{\sqrt\pi} \frac{\Gamma(\frac{q+1}{2})}{\Gamma(\frac{q+2}{2})} \left( \|e^{it\Delta/2}\psi^+\|_q^2 + \|e^{-it\Delta/2}\psi^-\|_q^2 \right)^{q/2} \\
& \qquad \leq \frac{2^{q/2}}{\sqrt\pi} \frac{\Gamma(\frac{q+1}{2})}{\Gamma(\frac{q+2}{2})} (2\pi)^{(d+2)/d} \mathcal S_d \left( \|\psi^+\|^2 + \|\psi^-\|^2 \right)^{q/2} \,.
\end{align*}
This proves the inequality $\leq$ in the corollary. The opposite inequality follows by choosing $\psi^+=\overline{\psi^-}$ to be almost maximizers for $\mathcal S_d$ and recalling that equality holds in Lemma \ref{equal} if $f=\overline g$.
\end{proof}


\section{Perturbative analysis}\label{sec:pert}

In this section we prove Proposition \ref{pert} which verifies the main assumption of Theorem \ref{main} provided Conjecture \ref{conj} holds. Let
$$
\psi_G(x):=e^{-x^2/2}
$$
and
$$
\mathcal S_d^G :=(2\pi)^{-(d+2)/d} \,\frac{\iint_{\R\times\R^d}|e^{it\Delta/2}\psi_G(x)|^{2+4/d}\,dx\,dt}{\|\psi_G\|^{2+4/d}} \,,
$$
so that Conjecture \ref{conj} is equivalent to the identity $\mathcal S_d = \mathcal S_d^G$. In view of this identity, Proposition \ref{pert} is an immediate consequence of Proposition \ref{prop:expansion-gaussian} below.

As explained in Remark \ref{gapnotstrict}, the \emph{non-strict} analogue of inequality \eqref{eq:gap} is obtained by glueing two Gaussians on the sphere that concentrate on two antipodal points. We now compute the next order of the `energy' of this trial function. Thus, for any $\epsilon>0$, consider the trial function
$$
f_\epsilon(\omega):=\chi(\omega_N)e^{-\frac{1-\omega_N}{\epsilon^2}}+\chi(-\omega_N)e^{-\frac{1+\omega_N}{\epsilon^2}} \quad\forall\omega\in\Sph^{N-1},
$$
where $\chi\in C_c^\ii(\R)$ is such that $\chi\equiv1$ in a neighborhood of $1$ and $\chi\equiv0$ in a neighborhood of $(-\infty,0]$. As $\epsilon\to0$, the functions $(f_\epsilon)$ concentrate on the north and south pole and the limiting profiles are, indeed, Gaussians.

\begin{proposition}\label{prop:expansion-gaussian}
We have
 \begin{equation}
  \log\frac{\int_{\R^N}|\check{f_\epsilon}|^q\,dx}{\|f_\epsilon\|^q}=\log\left( \frac{2^{q/2}}{\sqrt{\pi}}\frac{\Gamma\left(\frac{q+1}{2}\right)}{\Gamma\left(\frac{q+2}{2}\right)}\,\cS_{d}^G \right) +\frac{1}{4}\epsilon^2+o_{\epsilon\to0}(\epsilon^2) \,.
 \end{equation}
In particular, for all sufficiently small $\epsilon>0$,
$$
\frac{\int_{\R^N}|\check{f_\epsilon}|^q\,dx}{\|f_\epsilon\|^q}
> \frac{2^{q/2}}{\sqrt{\pi}}\frac{\Gamma\left(\frac{q+1}{2}\right)}{\Gamma\left(\frac{q+2}{2}\right)}\,\cS_{d}^G \,.
$$ 
\end{proposition}

An ingredient in the proof of this proposition is the following result about the simpler trial function
\begin{equation}
\label{eq:trialsingle}
g_\epsilon(x) := \chi(\omega_N)e^{-\frac{1-\omega_N}{\epsilon^2}} \quad\forall\omega\in\Sph^{N-1} \,,
\end{equation}
which concentrates only at the north pole. Similar results appear in \cite{ChrSha-12a,Shao-15} for $N=2,3$.

\begin{lemma}\label{singlebump}
We have
 \begin{equation}
  \log\frac{\int_{\R^N}|\check{g_\epsilon}|^q\,dx}{\|g_\epsilon\|^q}=\log \cS_{d}^G + \frac{1}{4}\epsilon^2+o_{\epsilon\to0}(\epsilon^2) \,.
 \end{equation}
\end{lemma}

Before proving the lemma, let us use it to give the

\begin{proof}[Proof of Proposition \ref{prop:expansion-gaussian}]
With $g_\epsilon$ from \eqref{eq:trialsingle} we shall show that
\begin{equation}
\label{eq:singlebump}
\frac{\int_{\R^N}|\check{f_\epsilon}|^q\,dx}{\|f_\epsilon\|^q}
= \frac{2^{q/2}}{\sqrt{\pi}}\frac{\Gamma\left(\frac{q+1}{2}\right)}{\Gamma\left(\frac{q+2}{2}\right)}
\frac{\int_{\R^N}|\check{g_\epsilon}|^q\,dx}{\|g_\epsilon\|^q}
+ \cO(\epsilon^4)\,. 
\end{equation}
This, together with Lemma \ref{singlebump}, implies the proposition.

Clearly, we have
\begin{equation}
\label{eq:singlebump1}
\| f_\epsilon \|^q = 2^{q/2} \| g_\epsilon \|^q \,.
\end{equation}
We also note the rough bound
\begin{equation}
\label{eq:singlebump1a}
\| g_\epsilon \|^q \geq c \epsilon^{d+2} \,.
\end{equation}
(We will prove something much more precise in the proof of Lemma \ref{singlebump}.) Moreover, let
\begin{equation}
\label{eq:phiepsilon}
\phi_\epsilon(x) := \frac{1}{(2\pi)^{N/2}} \int_{\R^d}e^{ix'\cdot\eta-\epsilon^{-2}(1-\sqrt{1-\epsilon^2|\eta|^2})(1+ix_N)}\frac{\chi(\sqrt{1-\epsilon^2|\eta|^2})}{\sqrt{1-\epsilon^2|\eta|^2}}\,d\eta
\end{equation}
and note that
$$
\epsilon^{-d}\check{f_\epsilon}(x'/\epsilon,x_N/\epsilon^2) = 2\re\left( e^{ix_N/\epsilon^2}\phi_\epsilon(x) \right)
\quad\text{and}\quad
\epsilon^{-d}\check{g_\epsilon}(x'/\epsilon,x_N/\epsilon^2) = e^{ix_N/\epsilon^2} \phi_\epsilon(x) \,.
$$
We claim that
\begin{equation}
\label{eq:singlebump2}
\int_{\R^N}\left|\re\left(e^{ix_N/\epsilon^2} \phi_\epsilon(x)\right) \right|^q\,dx=\frac{1}{2\pi}\int_{\R^d} \int_{-\pi}^{\pi}\left|\re \left(e^{i\theta}\phi_\epsilon(x)\right)\right|^q\,d\theta\,dx+\cO(\epsilon^4) \,.
\end{equation}
Since, as in \eqref{eq:gammafcn}, for any $a\in\C$,
$$
\frac{1}{2\pi} \int_0^{2\pi}\left|\re \left(e^{i\theta} a\right)\right|^q\,d\theta = \frac{1}{2\pi} \int_{-\pi}^{\pi} |\cos\theta|^q \,d\theta\ |a|^q = \frac{1}{\sqrt\pi} \frac{\Gamma(\frac{q+1}{2})}{\Gamma(\frac{q+2}{2})} \, |a|^q \,,
$$
we infer from \eqref{eq:singlebump2} that after scaling
$$
\int_{\R^N}|\check{f_\epsilon}|^q\,dx = \frac{2^q}{\sqrt\pi} \frac{\Gamma(\frac{q+1}{2})}{\Gamma(\frac{q+2}{2})}\int_{\R^N}|\check{g_\epsilon}|^q\,dx + \mathcal O(\epsilon^{4+(d+2)}) \,.
$$
This, together with \eqref{eq:singlebump1} and \eqref{eq:singlebump1a}, implies \eqref{eq:singlebump}.

Let us prove \eqref{eq:singlebump2}. We introduce the function
$$
a(x,\theta)= \left|\re \left(e^{i\theta}\phi_\epsilon(x)\right)\right|^q
\quad\forall (x,\theta)\in\R^N\times[-\pi,\pi] \,.
$$
Differentiating in $\theta$, we find that there is a $C>0$ such that for all $x\in\R^N$ and all $\epsilon>0$,
\begin{equation}
\label{eq:homoproof}
\|a(x,\cdot)\|_{L^\ii_\theta}+\|\partial_\theta^2 a(x,\cdot)\|_{L^\ii_\theta}\le C|\phi_\epsilon(x)|^q \,.
\end{equation}
As a consequence, we may expand $a$ as an absolutely convergent Fourier series
$$
a(x,\theta)=\sum_{n\in\Z}c_n(x)e^{in\theta},\quad\forall(x,\theta)\in\R^N\times[-\pi,\pi]
$$
with
$$
c_n(x):=\int_{-\pi}^{\pi}a(x,\theta)e^{-in\theta}\,\frac{d\theta}{2\pi} \,.
$$
By integration by parts and \eqref{eq:homoproof}, we find the bound
$$
|c_n(x)|\le\frac{C}{1+n^2}|\phi_\epsilon(x)|^q \,.
$$
By standard stationary phase arguments one can show that $\phi_\epsilon$ is bounded in $L^q(\R^N)$ uniformly for small $\epsilon>0$ and we obtain
$$
\int_{\R^N}\left|\re\left(e^{ix_N/\epsilon^2} \phi_\epsilon(x)\right)\right|^q dx
= \int_{\R^N} a(x,x_N/\epsilon^2) \,dx
=\sum_{n\in\Z}\int_{\R^N}e^{inx_N/\epsilon^2}c_n(x)\,dx \,.
$$
Hence, in order to prove \eqref{eq:singlebump2} we will prove
\begin{equation}
\label{eq:homogoal}
\sum_{n\neq0}\int_{\R^N}e^{inx_N/\epsilon^2}c_n(x)\,dx=\cO(\epsilon^4) \,.
\end{equation}
Integrating by parts, we have
$$
\int_{\R^N}e^{inx_N/\epsilon^2}c_n(x)\,dx=-\frac{\epsilon^4}{n^2}\int_{\R^N}e^{inx_N/\epsilon^2}\partial_{x_N}^2c_n(x)\,dx \,,
$$
and thus it is sufficient to bound $\|\partial_{x_N}^2c_n\|_{L^1_x}$ uniformly in $n$ and $\epsilon$. This bound again follows from stationary phase arguments, which imply that $\phi_\epsilon$, $\partial_{x_N}\phi_\epsilon$, and $\partial_{x_N}^2\phi_\epsilon$ are bounded in $L^q(\R^N)$, uniformly for small $\epsilon>0$. In this way we obtain \eqref{eq:homogoal} and therefore \eqref{eq:singlebump2} and \eqref{eq:singlebump}.
\end{proof}

Finally, we prove Lemma \ref{singlebump}. We will make repeated use of the Gaussian integrals
\begin{align}
\label{eq:gaussian-integral}
\int_{\R^d}e^{ix'\cdot\eta-\frac{s}{2}|\eta|^2}\,d\eta
& =\left(\frac{2\pi}{s}\right)^{d/2}e^{-\frac{1}{2s}|x'|^2} \,,\\
\label{eq:gaussian-integral2}
\int_{\R^d}e^{ix'\cdot\eta-\frac{s}{2}|\eta|^2}|\eta|^2\,d\eta
& =\left[\frac{d}{s}-\frac{|x'|^2}{s^2}\right]\left(\frac{2\pi}{s}\right)^{d/2}e^{-\frac{1}{2s}|x'|^2} \,, \\
\label{eq:gaussian-integral3}
\int_{\R^d}e^{ix'\cdot\eta-\frac{s}{2}|\eta|^2}|\eta|^4\,d\eta
& =\left[\frac{d(d+2)}{s^2}-\frac{2(d+2)|x'|^2}{s^3}+\frac{|x'|^4}{s^4}\right]\left(\frac{2\pi}{s}\right)^{d/2}e^{-\frac{1}{2s}|x'|^2} \,,
\end{align}
as well as the identities
\begin{equation}\label{eq:integrals-arctan}
\int_\R\frac{dx_N}{1+x_N^2}=\pi
\qquad\text{and}\qquad
\int_\R\frac{dx_N}{(1+x_N^2)^2}=\frac\pi2 \,.
\end{equation}

\begin{proof}[Proof of Lemma \ref{singlebump}]
With $\phi_\epsilon$ from \eqref{eq:phiepsilon} we note that
$$
\Psi(\epsilon^2):=\log\frac{\int_{\R^N}|\check{g_\epsilon}|^q\,dx}{\|g_\epsilon\|^q}=\log\frac{\int_{\R^N}|\phi_\epsilon|^q\,dx}{\|\epsilon^{-d/2} g_\epsilon\|^q} \,.
$$

We begin by studying $\|\epsilon^{-d/2} g_\epsilon\|^q$. Expanding
\begin{equation}
\label{eq:expansion1}
\epsilon^{-2}(1-\sqrt{1-\epsilon^2|\eta|^2})=\frac12|\eta|^2+\frac18\epsilon^2|\eta|^4+\cO(\epsilon^4|\eta|^6)
\end{equation}
and
\begin{equation}
\label{eq:expansion2}
\frac{\chi(\sqrt{1-\epsilon^2|\eta|^2})}{\sqrt{1-\epsilon^2|\eta|^2}}=1+\frac12\epsilon^2|\eta|^2+\cO(\epsilon^4|\eta|^4)
\end{equation}
(with the same expansion when $\chi$ is replaced by $\chi^2$), we obtain
\begin{align*}
\epsilon^{-d} \int_{\Sph^d}|g_\epsilon(\omega)|^2\,d\omega
& = \int_{\R^d}e^{-2\epsilon^{-2}(1-\sqrt{1-\epsilon^2|\eta|^2})}\frac{\chi(\sqrt{1-\epsilon^2|\eta|^2})^2}{\sqrt{1-\epsilon^2|\eta|^2}}\,d\eta \\
& = \int_{\R^d}e^{-|\eta|^2}\,d\eta+\epsilon^2\int_{\R^d}e^{-|\eta|^2}\left[\frac12 |\eta|^2-\frac14|\eta|^4\right]d\eta+\cO(\epsilon^4) \,.
\end{align*}
Using the formulas for Gaussian integrals \eqref{eq:gaussian-integral}, \eqref{eq:gaussian-integral2} and \eqref{eq:gaussian-integral3} we find that
\begin{equation}\label{eq:expansion-L2norm}
\epsilon^{-d}\int_{\Sph^d}|g_\epsilon(\omega)|^2\,d\omega=\pi^{d/2}+\frac{d(2-d)}{16}\pi^{d/2}\epsilon^2+\cO(\epsilon^4) \,.
\end{equation}
Note that the leading term coincides with
\begin{equation}
\label{eq:gaussian-L2norm}
\int_{\R^d} |\psi_G(x)|^2\,dx = \int_{\R^d} e^{-x^2}\,dx = \pi^{d/2} \,.
\end{equation}

Next, we discuss the asymptotics of $\|\phi_\epsilon\|_q$. Using expansions \eqref{eq:expansion1} and \eqref{eq:expansion2} and routine stationary phase arguments we obtain
\begin{align*}
\phi_\epsilon(x) & =\frac{1}{(2\pi)^{N/2}}\int_{\R^d}e^{ix'\cdot\eta-\frac12|\eta|^2(1+ix_N)}\,d\eta+o_{L^q_x(\R^N)}(1) \\
& = \frac{1}{(2\pi)^{1/2}}\left(\frac{1}{1+ix_N}\right)^{d/2}e^{-\frac{|x'|^2}{2(1+ix_N)}} +o_{L^q_x(\R^N)}(1)\,.
\end{align*}
The last identity used again \eqref{eq:gaussian-integral}. Thus,
\begin{align}
    \lim_{\epsilon^2\to0} \int_{\R^N}|\phi_\epsilon(x)|^q\,dx & =\frac{1}{(2\pi)^{q/2}} \, \int_{\R^N}(1+x_N^2)^{-dq/4}e^{-\frac{q|x'|^2}{2(1+x_N^2)}}\,dx\notag\\
    &=\frac{1}{(2\pi)^{q/2}} \, \iint_{\R\times\R^d}\left|\left(e^{it\Delta/2}\psi_G\right)(y)\right|^q \,dy\,dt \,, \label{eq:exact-expression-Lqnorm}
\end{align}
where the last identity used the Gaussian integral \eqref{eq:gaussian-integral}. Note that \eqref{eq:expansion-L2norm}, \eqref{eq:gaussian-L2norm} and \eqref{eq:exact-expression-Lqnorm} imply that
\begin{equation*}
  \lim_{\epsilon^2\to0}\Psi(\epsilon^2)=\log \cS_{d}^G \,,
\end{equation*}
which gives us the leading term in the lemma.

We claim that $\epsilon^2\mapsto \int_{\R^N}|\phi_\epsilon(x)|^q\,dx$ is differentiable at $\epsilon^2=0$ and that
\begin{equation}
\label{eq:compqnormderivative}
\partial_{\epsilon^2}\left(\int_{\R^N}|\phi_\epsilon|^q\,dx\right)|_{\epsilon^2=0} 
= \frac{1}{(2\pi)^{q/2}}\left(\frac{2\pi}{q}\right)^{d/2} \frac{\pi q d}{2(d+2)}\left[\frac12-\frac{d^2}{16}\right].
\end{equation}
We will discuss this below in some detail. Once this claim is shown, it is easy to complete the proof of the lemma. In fact, we note that
\begin{equation}
\label{eq:compderivative}
\partial_{\epsilon^2}\Psi(\epsilon^2)|_{\epsilon^2=0}=\frac{\partial_{\epsilon^2}\left(\int_{\R^N}|\phi_\epsilon|^q\,dx\right)|_{\epsilon^2=0}}{\lim_{\epsilon^2\to0}\int_{\R^N}|\phi_\epsilon|^q\,dx}
- \frac q2 \frac{\partial_{\epsilon^2}\left(\|\epsilon^{-d/2} g_\epsilon\|^2\right)|_{\epsilon^2=0}}{\lim_{\epsilon^2\to0}\|\epsilon^{-d/2} g_\epsilon\|^2} \,,
\end{equation}
and we recall from \eqref{eq:expansion-L2norm} that
\begin{equation}
\label{eq:compderivative2}
  \frac{\partial_{\epsilon^2}\left(\| \epsilon^{-d/2} g_\epsilon\|^2\right)|_{\epsilon^2=0}}{\lim_{\epsilon^2\to0}\|\epsilon^{-d/2} g_\epsilon\|^2}=\frac{d(2-d)}{16} \,.
\end{equation}
By the first identity in \eqref{eq:integrals-arctan} and the Gaussian integral \eqref{eq:gaussian-integral}, we compute from \eqref{eq:exact-expression-Lqnorm}
\begin{equation}
\lim_{\epsilon^2\to0}\int_{\R^N}|\phi_\epsilon|^q\,dx= \frac{1}{(2\pi)^{q/2}} \left(\frac{2\pi}{q}\right)^{d/2}\pi \,,
\end{equation}
which, combined with \eqref{eq:compqnormderivative}, gives
\begin{equation}
\label{eq:compderivative3}
  \frac{\partial_{\epsilon^2}\left(\int_{\R^N}|\phi_\epsilon|^q\,dx\right)|_{\epsilon^2=0}}{\lim_{\epsilon^2\to0}\int_{\R^N}|\phi_\epsilon|^q\,dx}=\frac{q}{2}\frac{d}{d+2}\left[\frac12-\frac{d^2}{16}\right]=\frac12-\frac{d^2}{16} \,.
\end{equation}
Inserting \eqref{eq:compderivative2} and \eqref{eq:compderivative3} into \eqref{eq:compderivative} leads to
  $$
  \partial_{\epsilon^2}\Psi(\epsilon^2)|_{\epsilon^2=0}=\frac12-\frac{d^2}{16}-\frac{q}{2}\frac{d(2-d)}{16}=\frac12-\frac{d^2}{16}-\frac{4-d^2}{16}=\frac14 \,,
  $$
which is the result stated in the lemma.
  
Thus, it remains to justify the claim \eqref{eq:compqnormderivative}. Using stationary phase arguments one can show that, for any $\sigma>0$,
$$
|\phi_\epsilon(x)| \leq C_\sigma (1+|x|)^{-d/2+\sigma}
\qquad\text{for all}\ x\in\R^N\,,\ \epsilon>0
$$
and
$$
\left|\re(\overline{\phi_\epsilon} \partial_{\epsilon^2} \phi_\epsilon(x))\right| \leq C_\sigma (1+|x|)^{-d+2\sigma}
\qquad\text{for all}\ x\in\R^N\,,\ \epsilon>0 \,.
$$
(The crucial point here is the real part which leads to a cancellation. Without the real part one can only obtain a similar bound with an additional factor of $|x_N|$, which is not good enough to prove differentiability.) These bounds imply by dominated convergence that $\epsilon^2\mapsto \int_{\R^N}|\phi_\epsilon(x)|^q\,dx$ is differentiable at any $\epsilon^2\geq 0$ and that
\begin{equation}
\label{eq:qnormderivative}
\partial_{\epsilon^2}\left(\int_{\R^N}|\phi_\epsilon|^q\,dx\right) = q\int_{\R^N} |\phi_\epsilon|^{q-2}\re\left(\overline{\phi_\epsilon}\partial_{\epsilon^2}\phi_\epsilon\right) dx \,.
\end{equation}
Recalling \eqref{eq:phiepsilon}, \eqref{eq:expansion1} and \eqref{eq:expansion2} we expand pointwise
\begin{align*}
\phi_\epsilon(x)& = \frac{1}{(2\pi)^{N/2}}\int_{\R^d}e^{ix'\cdot\eta-\frac12|\eta|^2(1+ix_N)}\,d\eta\\
  & \quad +\frac{\epsilon^2}{(2\pi)^{N/2}}\int_{\R^d}e^{ix'\cdot\eta-\frac12|\eta|^2(1+ix_N)}\left[\frac12 |\eta|^2-\frac18|\eta|^4(1+ix_N)\right]\,d\eta+o(\epsilon^2)\\
  &= \frac{1}{(2\pi)^{1/2}}\left(\frac{1}{1+ix_N}\right)^{d/2}e^{-\frac{|x'|^2}{2(1+ix_N)}} \\
&\quad +\frac{\epsilon^2}{(2\pi)^{1/2}}\left[\frac{d(2-d)}{8(1+ix_N)}+\frac{d|x'|^2}{4(1+ix_N)^2}-\frac{|x'|^4}{8(1+ix_N)^{3}}\right]\left(\frac{1}{1+ix_N}\right)^{d/2}e^{-\frac{|x'|^2}{2(1+ix_N)}} \\
 &\quad +o(\epsilon^2) \,.
\end{align*}
Here we used the Gaussian integral formulas \eqref{eq:gaussian-integral}, \eqref{eq:gaussian-integral2} and \eqref{eq:gaussian-integral3}. We obtain
\begin{align*}
|\phi_\epsilon|^{q-2}\re\left(\overline{\phi_\epsilon}\partial_{\epsilon^2}\phi_\epsilon \right)|_{\epsilon^2=0}
& = \left| \frac{1}{(2\pi)^{1/2}}\left(\frac{1}{1+ix_N}\right)^{d/2}e^{-\frac{|x'|^2}{2(1+ix_N)}} \right|^q \\
& \qquad \times \re \left[\frac{d(2-d)}{8(1+ix_N)}+\frac{d|x'|^2}{4(1+ix_N)^2}-\frac{|x'|^4}{8(1+ix_N)^{3}}\right] \\
& = \frac{1}{(2\pi)^{q/2}} \left(\frac1{1+x_N^2}\right)^{dq/4} e^{-\frac{q|x'|^2}{2(1+x_N^2)}} \\
& \qquad \times\left[\frac{d(2-d)}{8(1+x_N^2)}+\frac{d|x'|^2(1-x_N^2)}{4(1+x_N^2)^2}-\frac{|x'|^4(1-3x_N^2)}{8(1+x_N^2)^3}\right] .
\end{align*}
Finally, we integrate this identity over $x\in\R^d$ and recall \eqref{eq:qnormderivative}. We change variables $x'=(1+x_N^2)^{1/2}y$, compute Gaussian integrals and use \eqref{eq:integrals-arctan} to obtain
\begin{align*}
\partial_{\epsilon^2}\left(\int_{\R^N}|\phi_\epsilon|^q\,dx\right)|_{\epsilon^2=0}& =\frac{q}{(2\pi)^{q/2}}\iint_{\R\times\R^{N-1}}\,dx_N\,dy\,(1+x_N^2)^{-2}e^{-\frac q2 |y|^2}\times\\
& \qquad \times\left[\frac{d(2-d)}{8}+\frac{d|y|^2(1-x_N^2)}{4}-\frac{|y|^4(1-3x_N^2)}{8}\right] \\
& = \frac{q}{(2\pi)^{q/2}}\left(\frac{2\pi}{q}\right)^{d/2}\int_{\R}\,dx_N (1+x_N^2)^{-2}\times\\
& \qquad \times\left[\frac{d(2-d)}{8}+\frac{d^2(1-x_N^2)}{4q}-\frac{d(d+2)(1-3x_N^2)}{8q^2}\right] \\
& = \frac{1}{(2\pi)^{q/2}}\left(\frac{2\pi}{q}\right)^{d/2}\times\frac{\pi q d}{2(d+2)}\left[\frac12-\frac{d^2}{16}\right].
\end{align*}
This proves \eqref{eq:compqnormderivative}.
\end{proof}


\appendix

\section{Refinement of a perturbed Strichartz inequality}\label{sec:strichref}

In this appendix we show that the method from \cite{Tao-09,KilVis-book} can be used to prove the refinement of the perturbed Strichartz inequality in Proposition \ref{stref}. We actually prove it in the setting of elliptic-type phases as defined in \cite{TaoVarVeg-98}, thus we do not restrict ourselves to the case of the sphere with the function $T$. Instead, let $\Phi$ a smooth real function defined on a neighborhood of the origin in $\R^d$, satisfying $\text{Hess}\,\Phi(0)=\text{Id}$. We consider the general phase $\xi\mapsto\Phi(\xi)$ instead of $\xi\mapsto T(|\xi|^2)=1-\sqrt{1-|\xi|^2}$. We also recall that we denote the dimension by $d\geq 1$ and that
$$
q= 2+ 4/d \,.
$$

As we mentioned before, the main ingredient is a deep bilinear restriction estimate due to Tao. To state this result we introduce the notation
$$
Q\sim Q'
$$
for two dyadic cubes $Q,Q'\in\mathcal D$ to mean that they have the same side length and are not adjacent (i.e., their closures do not intersect), but their parents are adjacent. In the sequel, we use the shortcut notation for any $\psi\in L^2_x(\R^d)$ and any $(t,x)\in\R\times\R^d$,
$$
\Psi_Q(t,x) := \left( e^{-it\Phi(-i\nabla)}\psi_Q \right)(x),
$$
where we recall that $\hat{\psi_Q}:=\1_Q\hat{\psi}$.

\begin{theorem}\label{thm:bilinear-estimate}
Let $\frac{d+3}{d+1}<p<\frac{d+2}{d}$. There are $\epsilon>0$ and $C>0$ such that for all $\psi\in L^2(\R^d)$ with $\supp\hat\psi\subset \{|\xi|\leq\epsilon\}$ and for all $Q\sim Q'$ we have
\begin{equation}\label{eq:bilinear-estimate}
\left\| \Psi_Q\Psi_{Q'} \right\|_{L^p_{t,x}}\leq C |Q|^{1-\frac{d+2}{pd}}\|\psi_Q\|_{L^2_x}\|\psi_{Q'}\|_{L^2_x} \,.
\end{equation}
\end{theorem}
 
This theorem follows by a rather standard parabolic rescaling argument from Tao's sharp bilinear estimates on the paraboloid \cite{Tao-03} and from earlier bilinear estimates due to Tao--Vargas--Vega \cite{TaoVarVeg-98}. We present this derivation for the sake of completeness. We also remark that the assumption $p>\frac{d+3}{d+1}$ is sharp, but that for our purpose the inequality with any $p$ satisfying $p<\frac{d+2}{d}$ would be sufficient.

\begin{proof}
  Let $Q=\delta k+[0,\delta)^d$ and $Q'=\delta k'+[0,\delta)^d$ with $\delta\in 2^\Z$, $k,k'\in\Z^d$, $0<|k-k'|=\cO(1)$ and $\delta k=\cO(\epsilon)$, $\delta k'=\cO(\epsilon)$, $\delta=\cO(\epsilon)$. The parabolic rescaling leads to 
  $$\delta^{-d/2}\Psi_Q(t/\delta^2,x/\delta)=(2\pi)^{-d/2}\int_{\R^d}e^{ix\cdot\xi-it\delta^{-2}\Phi(\delta(k+\xi))}u_Q(\xi)\,d\xi,$$
  $$\delta^{-d/2}\Psi_{Q'}(t/\delta^2,x/\delta)=(2\pi)^{-d/2}\int_{\R^d}e^{ix\cdot\xi-it\delta^{-2}\Phi(\delta(k+\xi))}u_{Q'}(\xi)\,d\xi,$$
  where
  $$u_Q(\xi):=\delta^{d/2}\hat{\psi}(\delta(k+\xi))\1_{[0,1)^d}(\xi),$$
  $$u_{Q'}(\xi):=\delta^{d/2}\hat{\psi}(\delta(k+\xi))\1_{k'-k+[0,1)^d}(\xi).$$
  As a consequence, we may write 
  $$\left\|\Psi_Q\Psi_{Q'}\right\|_{L^p_{t,x}}=\delta^{d-\frac{d+2}{p}}\|Tu_QTu_{Q'}\|_{L^p_{t,x}},$$
  where
  $$Tg(t,x)=\int_{Q_0} e^{ix\cdot\xi-it\Phi_{\delta,k}(\xi)}g(\xi)\,d\xi,$$
  $Q_0$ is some big cube independent of $Q$ and $Q'$ containing both $[0,1)^d$ and $k'-k+[0,1)^d$, and
  $$\Phi_{\delta,k}(\xi)=\delta^{-2}\left[\Phi(\delta(k+\xi))-\Phi(\delta k)-\delta\nabla\Phi(\delta k)\cdot\xi\right]. 
  $$
  By a Taylor formula and the fact that $\delta k=\cO(\epsilon)$, $\delta=\cO(\epsilon)$, all the smooth norms of $\Phi_{\delta,k}$ are bounded uniformly in $(\delta,k)$ on $Q_0$. Furthermore,
  $$\text{Hess}\,\Phi_{\delta,k}-\text{Id}=\cO(\epsilon),$$
  also uniformly in $(\delta,k)$ on $Q_0$. We are thus in the setting the bilinear estimates of Tao \cite{Tao-03} (see the third remark at the end of the article), for elliptic-type compact surfaces as defined in \cite[Sec. 2]{TaoVarVeg-98}. We
   deduce that if $\epsilon>0$ is small enough, there exists $C>0$ independent of $Q$, $Q'$ and $g$ such that
  $$
  \left\|\Psi_Q\Psi_{Q'}\right\|_{L^p_{t,x}}\le C\delta^{d-\frac{d+2}{p}}\|u_Q\|_{L^2(\R^d)}\|u_{Q'}\|_{L^2(\R^d)}.
  $$
  Undoing all the change of variables that we performed, we find that
  $$
  \|u_Q\|_{L^2(\R^d)}\|u_{Q'}\|_{L^2(\R^d)}=\|\psi_Q\|_{L^2_x}\|\psi_{Q'}\|_{L^2_x},
  $$
  which implies the desired estimate.  
 \end{proof}

The next ingredient in the proof of Proposition \ref{steintomref} is the following improvement over the triangle inequality.

\begin{lemma}\label{lem:triangle-ineq-improved}
For $\epsilon>0$ small enough, there is a $C>0$ such that for all $\psi\in L^2(\R^d)$ with $\supp\hat\psi\subset\{|\xi|\leq \epsilon\}$,
\begin{equation}\label{eq:triangle-ineq-improved}
\left\|\sum_{Q\sim Q'} \Psi_Q\Psi_{Q'} \right\|_{L^{q/2}_{t,x}}^{q^*}
\leq C \sum_{Q\sim Q'}\left\|\Psi_Q\Psi_{Q'}\right\|_{L^{q/2}_{t,x}}^{q^*}
\end{equation}
with $q^*:=\min\{q/2,(q/2)'\}$.
\end{lemma}

 \begin{proof}[Proof of Lemma \ref{lem:triangle-ineq-improved}]
 We apply the result of Tao--Vargas--Vega \cite{TaoVarVeg-98}, or more precisely the version of \cite[Lem.~A.9 \& Proof of Prop.~4.24]{KilVis-book}. The space-time Fourier transform $\F_{t,x}$  of $\Psi_Q\Psi_{Q'}$ satisfies
 $$\supp\,\F_{t,x}[\Psi_Q\Psi_{Q'}]\subset\{(\eta+\eta',\Phi(\eta)+\Phi(\eta')),\ \eta\in Q,\ \eta'\in Q'\}.$$
 We include this last set into a similar parallelepided as in Killip-Visan [Proof of Prop. 4.24], which is then enough to obtain orthogonality. Taylor expansions leads to the formula
 $$\Phi(\eta)+\Phi(\eta')=2\Phi\left(\frac{\eta+\eta'}{2}\right)+a(\eta,\eta')|\eta-\eta'|^2,$$
 \begin{multline*}
    \Phi\left(\frac{\eta+\eta'}{2}\right)=\Phi\left(\frac{c(Q+Q')}{2}\right)+\frac{1}{2}\nabla\Phi\left(\frac{c(Q+Q')}{2}\right)\cdot(\eta+\eta'-c(Q+Q'))\\
    +b(\eta+\eta',c(Q+Q'))|\eta+\eta'-c(Q+Q')|^2,
 \end{multline*}
 where $c(Q+Q')$ denotes the center of the cube $Q+Q'$, for two functions $a$ and $b$ satisfying
 $$3/8\le a(\eta,\eta')\le1/8,\quad 3/16\le b(\eta+\eta',c(Q+Q'))\le 1/16,$$
 assuming that $\epsilon>0$ is small enough. We deduce that 
 $$\supp\,\F_{t,x}[\Psi_Q\Psi_{Q'}]\subset R(Q+Q'),$$
 where
 $$R(Q'')=\left\{(\eta,\omega),\ \eta\in Q'',\ \frac{869}{64}\le\frac{\omega-2\Phi\left(\frac12 c(Q'')\right)-\nabla\Phi\left(\frac12 c(Q'')\right)\cdot(\eta-c(Q''))}{(\text{diam}\, Q'')^2}\le\frac12\right\}.$$
 Again by a Taylor formula, we have 
 $$\frac34|c(Q'')-\eta|^2\le\Phi\left(\frac12 c(Q'')\right)-\Phi\left(\frac12 \eta\right)-\nabla\Phi\left(\frac12 c(Q'')\right)\cdot\frac12(c(Q'')-\eta)\le\frac14|c(Q'')-\eta|^2.$$
 We deduce that for any $(\eta,\omega)\in R(Q'')$, we have
 $$\left(\frac{869}{64}+\frac32\right)(\text{diam}\, Q'')^2\le \omega-\frac12\Phi\left(\frac12 \eta\right)\le (\text{diam}\, Q'')^2.$$
 This means that if two pairs of close cubes $Q\sim Q'$, $\tilde{Q}\sim\tilde{Q}'$ are such that $R(Q+Q')$ and $R(\tilde{Q}+\tilde{Q}')$ intersect, they must have a similar diameter. The same holds for the dilates $(1+\alpha) R(Q+Q')$ for some small $\alpha$, by the same argument. If the diameters are in a finite number, the cubes are also in a finite number since their centers verify 
 $$|c(Q'')-c(\tilde{Q}'')|\le|c(Q'')-\eta|+|c(\tilde{Q}'')-\eta|\le\text{diam}\, Q''.$$
 We are thus in the same situation as in Killip-Visan [Proof of Prop. 4.24], and Lemma \ref{lem:triangle-ineq-improved} follows.
\end{proof} 

As a final ingredient in the proof of Proposition \ref{steintomref} we cite a bound of sums of local norms over dyadic cubes in terms a global norm. For a simple proof we refer to \cite[Proof of Thm. A.1]{Tao-09}; see also \cite[Thm. 1.3]{BegVar-07} and \cite[Proof of Prop. 4.24]{KilVis-book}.

\begin{lemma}\label{lem:uniform-bound-sum-tao}
Let $d\geq 1$ and $1<\mu<\nu$. Then there is a constant $C_{d,\mu,\nu}$ such that for all $f\in L^\mu(\R^d)$,
$$
\left( \sum_{Q\in\mathcal D} |Q|^{-\nu/\mu'} \|f\|_{L^1(Q)}^\nu \right)^{1/\nu} \leq C_{d,\mu,\nu} \|f\|_{L^\mu(\R^d)} \,.
$$
\end{lemma}

After these preliminaries we are in position to give the

\begin{proof}[Proof of Proposition \ref{stref}]
We follow rather closely Tao's arguments \cite[Proof of Thm. A.1]{Tao-09}; see also Killip-Visan \cite[proof of Prop. 4.24]{KilVis-book}. We observe that for any $\xi,\xi'\in\R^d$ there is a pair of cubes $Q,Q'\in\mathcal D$ with $Q\sim Q'$ such that $\xi\in Q$ and $\xi'\in Q'$. Consequently, if we let
$$
\Psi(t,x) := \left( e^{-it\Phi(-i\nabla)}\psi\right)(x),
$$
we find
$$
\Psi^2=\sum_{Q\sim Q'} \Psi_Q \Psi_{Q'} \,.
$$
Therefore Lemma \ref{lem:triangle-ineq-improved} yields
\begin{align}\label{eq:strefproof}
\left\| \Psi \right\|_{L^q_{t,x}}^2 & = \left\| \Psi^2 \right\|_{L^{q/2}_{t,x}}
= \left\| \sum_{Q\sim Q'} \Psi_Q \Psi_{Q'} \right\|_{L^{q/2}_{t,x}} \leq C^{1/q^*} \left( \sum_{Q\sim Q'} \left\|\Psi_Q \Psi_{Q'} \right\|_{L^{q/2}_{t,x}}^{q^*} \right)^{1/q^*} \notag \\
& \leq C^{1/q^*} \sup_{Q\sim Q'} \left\|\Psi_Q \Psi_{Q'} \right\|_{L^{q/2}_{t,x}}^{q^*-r} \left( \sum_{Q\sim Q'} \left\|\Psi_Q \Psi_{Q'} \right\|_{L^{q/2}_{t,x}}^r \right)^{1/q^*}
\end{align}
for every $r\leq q^*$. We will later choose $r>1$. We now estimate $\left\|\Psi_Q \Psi_{Q'} \right\|_{L^{q/2}_{t,x}}$ in two different ways. They both rely on the bilinear estimate from Theorem \ref{thm:bilinear-estimate}. Since
$$
\|\psi_Q\|_{L^2_x}\leq \|\psi\|_{L^2_x},
$$
the bilinear estimate \eqref{eq:bilinear-estimate} implies that for all $\frac{N+2}{N}<p<\frac{N+1}{N-1}$ and all $Q\sim Q'$,
\begin{align*}
\left\|\Psi_Q \Psi_{Q'} \right\|_{L^{q/2}_{t,x}}
& \leq \left(|Q|^{-1} \left\|\Psi_Q \Psi_{Q'} \right\|_{L^\infty_{t,x}}\right)^{1-\frac{2p}{q}} \left( |Q|^{q/(2p)-1} \left\|\Psi_Q \Psi_{Q'} \right\|_{L^{p}_{t,x}} \right)^{\frac{2p}{q}}\nonumber \\
& \leq C_{N,p,\epsilon}^{\frac{2p}{q}} \left( |Q|^{-1/2} \left\|\Psi_Q\right\|_{L^{\ii}_{t,x}}\right)^{1-\frac{2p}{q}} \left(|Q'|^{-1/2} \left\| \Psi_{Q'}\right\|_{L^{\ii}_{t,x}}\right)^{1-\frac{2p}{q}} \|\psi\|_{L^2_x}^{\frac{4p}{q}} \,.
 \end{align*}
This bound implies
\begin{equation}
\label{eq:Linfty-estimate-bilinear}
\sup_{Q\sim Q'} \left\|\Psi_Q \Psi_{Q'} \right\|_{L^{q/2}_{t,x}}
\leq C_{N,p,\epsilon}^{\frac{2p}{q}}  \left( \sup_{Q''\in\mathcal D}  |Q''|^{-1/2} \left\|\Psi_{Q''}\right\|_{L^{\ii}_{t,x}}\right)^{2(1-\frac{2p}{q})} \|\psi\|_{L^2_x}^{\frac{4p}{q}} \,.
\end{equation}

On the other hand, one can also interpolate the bilinear estimate \eqref{eq:bilinear-estimate} with the trivial estimate
$$
\|\Psi_Q \Psi_{Q'}\|_{L^\ii_{t,x}}\leq (2\pi)^{-d} \|\hat\psi_Q\|_{L^1_\xi}\|\hat\psi_{Q'}\|_{L^1_\xi}
$$
to obtain 
\begin{equation*}
\|\Psi_Q \Psi_{Q'}\|_{L^{q/2}_{t,x}}\leq C_{N,\epsilon}' |Q|^{1-\frac2s}\|\hat\psi_Q\|_{L^s_\xi}\|\hat\psi_{Q'}\|_{L^s_\xi}
\end{equation*}
for some $1<s<2$ (whose value is not important here). This implies that 
\begin{align}\label{eq:second-estimate-Lq-norm}
\sum_{Q\sim Q'}\left\|\Psi_Q \Psi_{Q'} \right\|_{L^{q/2}_{t,x}}^r
& \leq C\sum_{Q\sim Q'}\left[|Q|^{1-\frac2s}\|\hat\psi_Q\|_{L^s_\xi}\|\hat\psi_{Q'}\|_{L^s_\xi}\right]^r 
\leq  C\sum_{Q\sim Q'}\left[|Q|^{1-\frac2s}\|\hat\psi_Q\|_{L^s_\xi}^2\right]^r \notag \\
& = C' \sum_{Q}\left[|Q|^{1-\frac2s}\|\hat\psi_Q\|_{L^s_\xi}^2\right]^r \,.
\end{align}
In the last equality we used the fact that the number of $Q'\in\mathcal D$ satisfying $Q'\sim Q$ is finite and independent of $Q$. Finally, according to Lemma \ref{lem:uniform-bound-sum-tao} (with $f=|\hat\psi|^s$, $\mu=2/s$ and $\nu=2r/s$; note $1<\mu<\nu$ since $s<2$ and $r>1$), the right side of \eqref{eq:second-estimate-Lq-norm} is bounded by a constant times $\|\hat\psi\|_{L^2_\xi}^{2r} = \|\psi\|_{L^2_x}^{2r}$. Combining this with \eqref{eq:strefproof} and \eqref{eq:Linfty-estimate-bilinear} completes the proof of Proposition \ref{stref}.
\end{proof}


\end{document}